\documentclass{amsart}

\usepackage{amsthm,amsfonts,amsmath,amssymb}
\usepackage{hyperref}
\usepackage[english]{babel}
\usepackage{verbatim}
\usepackage{multirow}
\usepackage[all]{xy}
\usepackage{tikz-cd}
\usepackage{float}
\usepackage{caption} 
\theoremstyle{plain}
\newtheorem{lemma}{Lemma}[section]
\newtheorem{theorem}[lemma]{Theorem}
\newtheorem{proposition}[lemma]{Proposition}
\newtheorem{corollary}[lemma]{Corollary}
\newtheorem{conjecture}[lemma]{Conjecture}

\theoremstyle{definition}

\theoremstyle{remark}
\newtheorem{remark}[lemma]{Remark}
\newtheorem{example}[lemma]{Example}

\DeclareMathOperator{\supp}{supp}

\DeclareMathOperator{\diag}{diag}

\DeclareMathOperator{\Hom}{Hom}

\DeclareMathOperator{\GL}{GL}
\DeclareMathOperator{\Sp}{Sp}
\DeclareMathOperator{\SO}{SO}
\DeclareMathOperator{\SL}{SL}
\DeclareMathOperator{\Spin}{Spin}

\renewcommand{\setminus}{\smallsetminus}
\renewcommand{\emptyset}{\varnothing}
\renewcommand{\leq}{\leqslant}
\renewcommand{\geq}{\geqslant}

\newcommand{\sfA}{\mathsf A}
\newcommand{\sfB}{\mathsf B}
\newcommand{\sfC}{\mathsf C}
\newcommand{\sfBC}{\mathsf{BC}}
\newcommand{\sfD}{\mathsf D}
\newcommand{\sfE}{\mathsf E}
\newcommand{\sfF}{\mathsf F}
\newcommand{\sfG}{\mathsf G}

\newcommand{\calR}{\mathcal R}

\newcommand{\CC}{\mathbb C}
\newcommand{\HH}{\mathbb H}
\newcommand{\NN}{\mathbb N}
\newcommand{\OO}{\mathbb O}
\newcommand{\PP}{\mathbb P}
\newcommand{\QQ}{\mathbb Q}
\newcommand{\RR}{\mathbb R}
\newcommand{\ZZ}{\mathbb Z}

\newcommand{\rmN}{\mathrm N}
\newcommand{\rmZ}{\mathrm Z}
\newcommand{\rmn}{\mathrm n}
\newcommand{\rmsc}{\mathrm{sc}}

\newcommand{\calM}{\mathcal M}
\newcommand{\calX}{\mathcal X}

\title[Multiplication of spherical functions]{On the multiplication of spherical functions\\ of reductive spherical pairs of type A}

\author{Paolo Bravi}
\address{Dipartimento di Matematica, Sapienza Universit\`a di Roma, Piazzale A. Moro 2, 00185 Roma, Italy}
\email{bravi@mat.uniroma1.it}

\author{Jacopo Gandini}
\address{Dipartimento di Matematica, Universit\`a di Bologna, Piazza di Porta San Donato 5, 40126 Bologna, Italy}
\email{jacopo.gandini@unibo.it}

\begin{document}

\maketitle

\begin{abstract}
Let $G$ be a simple complex algebraic group and let $K \subset G$ be a reductive subgroup such that the coordinate ring of $G/K$ is a multiplicity free $G$-module. We consider the $G$-algebra structure of $\CC[G/K]$, and study the decomposition into irreducible summands of the product of irreducible $G$-submodules in $\CC[G/K]$. When the spherical roots of $G/K$ generate a root system of type $\sfA$ we propose a conjectural decomposition rule, which relies on a conjecture of Stanley on the multiplication of Jack symmetric functions. With the exception of one case, we show that the rule holds true whenever the root system generated by the spherical roots of $G/K$ is a direct sum of subsystems of rank one.
\end{abstract}

\section{Introduction}
Let $G$ be a connected semisimple complex algebraic group, and let $K \subset G$ be a reductive subgroup. Recall that the affine variety $G/K$ is called \textit{spherical} if it contains an open orbit under a Borel subgroup of $G$. If this is the case, we will also say that $K$ is a \textit{spherical subgroup}, and that $(G,K)$ is a \textit{reductive spherical pair}. A fundamental example is that of the \textit{symmetric varieties}, namely when $K$ is the set of fixed points of an algebraic involution of $G$.

By a result of Vinberg and Kimelfeld \cite{VK}, reductive spherical subgroups are characterized by the property that the coordinate ring $\CC[G/K]$ is a \textit{multiplicity free} $G$-module (that is, every isotypic component is irreducible). 

Fix a maximal torus $T \subset G$ and a Borel subgroup of $G$ containing $T$, and let $\calX(T)^+$ be the monoid of dominant weights of $G$. For $\lambda \in \calX(T)^+$, we denote by $V(\lambda)$ the irreducible $G$-module of highest weight $\lambda$.

Suppose now that $X$ is an affine spherical homogeneous $G$-variety. Being multiplicity free, the $G$-module structure of the coordinate ring of $X$ is completely captured by the associated \textit{weight monoid}, that is the submonoid $\Lambda^+_X \subset \calX(T)^+$ formed by the highest weights of $\CC[X]$. Such monoid is well understood, and can be combinatorially described (see e.g. \cite{BP}). On the other hand, the description of the $G$-algebra structure of $\CC[X]$ is still a widely open problem.

To be more precise, for $\lambda \in \Lambda_X^+$ denote by $E_X(\lambda) \simeq V(\lambda)$ the irreducible component in $\CC[X]$ of highest weight $\lambda$. Then we have the decomposition
$$
	\CC[X] \; \; = \bigoplus_{\lambda \in \Lambda_X^+} E_X(\lambda) 
$$
Given $\lambda, \mu \in \Lambda_X^+$, let $E_X(\lambda) \cdot E_X(\mu)$ denote the submodule of $\CC[X]$ generated by the products $fg$ with $f \in E_X(\lambda)$ and $g \in E_X(\mu)$. We will deal with the following problem: for which $\nu \in \Lambda^+_X$ does $E_X(\nu)$ appear inside $E_X(\lambda) \cdot E_X(\mu)$?

Clearly there are some bounds on the range of weights which can occur inside $E_X(\lambda) \cdot E_X(\mu)$. For instance, if $\nu$ is such a weight, then $V(\nu)$ must appear as a submodule inside the tensor product $V(\lambda) \otimes V(\mu)$.

A more subtle bound comes from the \textit{spherical roots} of $X$.
Let indeed $\mathcal M_X^\rmn \subset \calX(T)$ be the submonoid generated by the set of differences
$$
	\{ \lambda+\mu-\nu \; | \; \lambda, \mu, \nu \in \Lambda^+_X \text{ and } E_X(\nu) \subset E_X(\lambda) \cdot E_X(\mu) \}
$$
By a fundamental theorem of Knop \cite{Kn}, together with a recent result of Avdeev and Cupit-Foutou \cite{ACF}, the monoid $\mathcal M_X^\rmn$ is free. Let $\Delta_X^\rmn$ be the set of free generators of $\mathcal M_X^\rmn$, called the \textit{spherical roots} of $X$ (or, more precisely, the $\rmn$-\textit{spherical roots}, in the terminology of Section \ref{sec:sph}). Then $\Delta_X^\rmn$ is the base of a root system $\Phi^\rmn_X$, whose rank is called the \textit{rank} of $X$. This root system, together with its Weyl group, encodes important information on the geometry of $X$ and of its equivariant embeddings. When $X$ is a symmetric variety, the root system $\Phi^\rmn_X$ is essentially the restricted root system attached to $X$.

Similarly to the weight monoid $\Lambda^+_X$, the set of spherical roots $\Delta_X^\rmn$ is also well understood (see e.g. \cite{BP}). In the general context of spherical varieties, both the weight monoid and the spherical roots are a fundamental part of the combinatorial invariants attached by Luna to a spherical variety in his seminal paper \cite{Lu1}.

Going back to the problem of decomposing the product of irreducible $G$-modules inside $\CC[X]$, by the very definition of the spherical roots we see the implication
$$
E_X(\nu) \subset E_X(\lambda) \cdot E_X(\mu) \Longrightarrow
\left\{ \begin{array}{c}
V(\nu) \subset V(\lambda) \otimes V(\mu)  \phantom{\Big|} \\
	\lambda + \mu - \nu \in \NN \Delta_X^\rmn
\end{array} \right.
$$
In several cases, the condition on the right hand side seems indeed close to give a characterization of the irreducible constituents of $E_X(\lambda) \cdot E_X(\mu)$. However there are counterexamples (see e.g. Example \ref{ex:counterexample}), thus we are lead to consider other possible descriptions.

In particular, in the present paper we will focus on the case of the affine spherical homogeneous spaces $X = G/K$ for a simple group $G$, whose spherical roots form a root system of type $\sfA$. The classification of the reductive spherical subgroups for a simple group was first obtained by Kr\"amer \cite{Kra}, and more recently in arbitrary characteristic by Knop and R\"ohrle \cite{KR}. For a complete list of the cases we will deal with, see Table \ref{tab:sym_pairs_A} (the symmetric cases) and Table \ref{tab:sph_pairs_A} (the non-symmetric cases). 

When $X$ is a symmetric variety with restricted root system of type $\sfA$, a general conjectural rule for decomposing the product of two irreducible components of $\CC[X]$ follows from a long standing conjecture of Stanley \cite{St} on the multiplication of Jack symmetric functions (see Conjecture \ref{conj:stanley_cor}). Such conjectural description holds true when the rank of $X$ is one, as a consequence of a Pieri rule for the multiplication of Jack symmetric functions proved by Stanley himself \cite{St}.

The connection between Stanley's context and ours stems from the well-known fact that specializing the Jack symmetric functions to a suitable parameter defined by $X$ (which is essentially the multiplicity of the restricted roots) and setting the extra variables to $0$ yields indeed a basis for the space $\CC[X]^K$ of the \textit{spherical functions} on $X$. This point of view was already exploited by Graham and Hunziker \cite{GH}, who dealt with a problem similar to the one considered in this paper in the case of the symmetric varieties with restricted root system of type $\sfA$.

Our main contribution with the present paper is to show how a general conjectural rule for decomposing the product of two irreducible constituents in $\CC[X]$ also follows from Stanley's mentioned conjecture for all the non-symmetric affine spherical homogeneous varieties whose spherical roots generate a root system of type $\sfA$. When this root system is a direct sum of subsystems of rank one, which happens in most cases, we will see (with the exception of one case) that such description holds true thanks to Stanley's Pieri rule for Jack symmetric functions.

To state the mentioned decomposition rule, we will introduce a different normalization for the spherical roots of $X$ and will define a slight variation of $\Phi^\rmn_X$, where we replace some of the spherical roots of $X$ with their half. More precisely, we will associate to $X$ a \textit{based root datum} $\calR_X$ (see Proposition \ref{prop:root_datum} for its definition), having the same Weyl group of $\Phi_X^\rmn$ and with the property that the weight monoid $\Lambda^+_X$ can be naturally regarded as a submonoid of the dominant weights of $\calR_X$.

Let $G_X$ be the complex reductive group associated to $\calR_X$, and for $\lambda \in \Lambda^+_X$ denote by $V_X(\lambda)$ the irreducible $G_X$-module associated to $\lambda$ (regarded as a dominant weight for $G_X$). Then we have the following conjecture (see Conjecture \ref{conj:affine_spherical_A}).

\begin{conjecture} \label{conj:intro}
Let $X$ be an affine spherical homogeneous variety for a simple group $G$ such that $\Phi_X^\rmn$ is of type $\sfA$. Let $\lambda, \mu, \nu \in \Lambda_X^+$, then
$$
E_X(\nu) \subset E_X(\lambda) \cdot E_X(\mu) \Longleftrightarrow
\left\{ \begin{array}{c}
V_X(\nu) \subset V_X(\lambda) \otimes V_X(\mu)  \phantom{\Big|} \\
\lambda + \mu - \nu \in \NN \Delta_X^\rmn
\end{array} \right.
$$
\end{conjecture}

Our main theorem is the following (see Proposition \ref{prop:indec}, Corollary \ref{cor:decomposizione} and Proposition \ref{prop:altri}).

\begin{theorem}
Let $X$ be an affine spherical homogeneous variety for a simple group $G$. Suppose that $\Phi_X^\rmn$ is of type $\sfA$, and that $X \neq \sfF_4/\Spin(9)$.
\begin{itemize}
	\item[i)] If $\Phi_X^\rmn$ is a direct sum of rank one subsystems, then Conjecture \ref{conj:intro} holds for $X$.
	\item[ii)] If Stanley's conjecture on the multiplication of Jack symmetric functions holds true, then Conjecture \ref{conj:intro} holds true for $X$.
\end{itemize}
\end{theorem}

In the case of $\sfF_4/\Spin(9)$, which remains excluded from the previous theorem, Conjecture \ref{conj:intro} is based on computational experiments.

The strategy that we will adopt to prove the previous theorem is somehow homogeneous in the various cases of Table \ref{tab:sph_pairs_A}. Indeed we will show that in all cases (with the exception of $\sfF_4/\Spin(9)$, and three other cases that can be easily reduced to some other case) we can find a reductive overgroup $\widehat G \supset G$ with a symmetric $\widehat G$-variety $Y$ and a reductive subgroup $H \subset G$ with a symmetric $H$-variety $Z$, respectively endowed with a $G$-equivariant isomorphism and with a $H$-equivariant embedding
$$
	Z \hookrightarrow X \stackrel{\sim}{\longrightarrow} Y .
$$
Both the symmetric varieties $Y$ and $Z$ will have a restricted root system of type $\sfA$, and we will see that the previous maps induce an isogeny from the based root datum $\calR_X$ to the direct sum $\calR_Y \oplus \calR_Z$ (see Proposition \ref{prop:isogeny}).

This will allow us to show that the validity of Conjecture \ref{conj:intro} in the various cases follows from the analogous statement for the symmetric varieties with restricted root system of type $\sfA$, which in turn follows from Stanley's conjecture, and from Stanley's Pieri rule when $\Phi_X^\rmn$ is a direct sum of rank one root subsystems.

The fact that the $G$-action on $X$ can be extended to a $\widehat G$-action making $X$ into a symmetric $\widehat G$-variety will be basically a consequence of a suitable factorization of $\widehat G$ as a product of two reductive subgroups (one of them being $G$, and the other one a symmetric subgroup of $\widehat G$ with the desired property). Such factorizations were first classified by Onishchik \cite{On}, and in arbitrary characteristic by Liebeck, Saxl and Seitz \cite{LSS}. 

In the last section we will also show that the same techinque described above can also be applied in the case of the spherical variety $\SO(2n+1)/\GL(n)$ (whose spherical root system is not of type $\sfA$), to decompose its root datum into the direct sum of those associated with two symmetric varieties of Hermitian type (see Theorem \ref{teo:isogeny-dec}). This leads us to consider the decomposition problem also for another class of reductive spherical pairs $(G,K)$, namely when $K$ is a spherical Levi subgroup of $G$, in which case $\Phi_{G/K}^\rmn$ is of type $\sfB$ or $\sfC$. Here the problem seems quite unexplored, even in the symmetric case. We show that the case of the spherical Levi subgroups reduces to that of the symmetric Levi subgroups (that is, to the case of the symmetric varieties of Hermitian type), and we formulate a conjectural rule for the decomposition of the product of two irreducible components in $\CC[G/K]$ in this case as well (see Conjecture \ref{conj:hermitiano}). \\

\textit{Acknowledgments.} We thank the anonymous referees for useful comments.

%%%%%%%%%%%%%%%%%%%%%%%%%%%%%%%%%%%%%%%
%%%%%%%%%%%%%%%%%%%%%%%%%%%%%%%%%%%%%%%
\section{Notation and preliminaries}\label{sec:gen}
%%%%%%%%%%%%%%%%%%%%%%%%%%%%%%%%%%%%%%%
%%%%%%%%%%%%%%%%%%%%%%%%%%%%%%%%%%%%%%%

If $H$ is a connected algebraic group, we will denote by $\calX(H)$ its character lattice. If $H$ acts on a vector space $V$, we will denote by $V^H$ the invariant subspace and by $V^{(H)}$ the semi-invariant subset, that is the union of the semi-invariant subspaces 
$$
	V^{(H)}_\chi = \{v \in V \; | \; h.v = \chi(h)v \quad \forall h \in H\}
$$
with $\chi \in \calX(H)$. If $\Xi$ is a lattice we will denote by $\Xi^\vee$ its dual lattice, namely  $\Xi^\vee= \Hom_\ZZ(\Xi, \ZZ)$.

Let $G$ be a semisimple complex algebraic group. If $K \subset G$, we will denote by $\rmN_G(K)$ the normalizer of $K$ in $G$, and by $Z_G(K)$ the centralizer of $K$ in $G$. Fix a maximal torus $T \subset G$ and a Borel subgroup $B$ contianing $T$. We denote by $\Phi$ the root system of $G$ associated to $T$, and by $\Delta$ the base defined by $B$ and by $\Phi = \Phi^+ \sqcup \Phi^-$ the corresponding decomposition into positive and negative roots. When $\Phi$ is an irreducible root system of rank $n$, we will order the set of simple roots $\Delta = \{\alpha_1, \ldots, \alpha_n\}$ following Bourbaki's notation.

The monoid of dominant weights of $G$ will be denoted by $\calX(T)^+$. Given $\lambda \in \calX(T)^+$, the irreducible representation of $G$ of highest weight $\lambda$ will be denoted by $V_G(\lambda)$, or simply by $V(\lambda)$ when no ambiguity is possible. 

Let $G/K$ be an affine spherical variety. Equivalently, $K$ is a reductive group and the coordinate ring $\CC[G/K]$ is multiplicity free as a $G$-module. By the multiplicity-free property, the description of the $G$-module structure of $\CC[G/K]$ is equivalent to the description of the monoid of the \textit{spherical weights}
$$
	\Lambda^+_{G/K} = \{\lambda \in \calX(T)^+ \; | \; V(\lambda)^K \neq 0\}.
$$
Since $K$ is reductive, notice that $V(\lambda)$ admits a non-zero $K$-invariant vector if and only if the same holds for its dual $V(\lambda)^*$. Therefore, with this notation, we have the decomposition
$$
	\CC[G/K] = \CC[G]^K \simeq \bigoplus_{\lambda \in \calX(T)^+} V(\lambda)^* \otimes V(\lambda)^K \simeq \bigoplus_{\lambda \in \Lambda_{G/K}^+} V(\lambda).
$$
For every $\lambda \in \Lambda_{G/K}^+$ we denote by $E_{G/K}(\lambda)\simeq V(\lambda)$ the corresponding isotypic (irreducible) component in $\CC[G/K]$. 

The previous description can be generalized from the realm of $K$-invariant functions to that of $K$-semi-invariant functions, with respect to the right action of $G$ on itself. Indeed, for every character $\chi \in \calX(K)$ the eigenspace $\CC[G]^{(K)}_\chi$ is also a multiplicity free $G$-module. Inside $\CC[G]^{(K)}_\chi$, every $\lambda \in \Lambda^+$ with $\big(V(\lambda)^*\big)^{(K)}_\chi \neq 0$ yields an isotypic component $E_{G/K}(\lambda,\chi) \simeq V(\lambda)$. In this setting, the relevant weight monoid is that of the \textit{quasi-spherical weights}
$$
	\Omega^+_{G/K} = \{\lambda \in \calX(T)^+ \; | \; V(\lambda)^{(K)} \neq 0\}.
$$ 

The algebra structure of the invariant ring $\CC[G]^K$ is encoded by the associated \textit{spherical functions}. For every $\lambda \in \Lambda^+_{G/K}$, fix a non-zero element $f_\lambda \in E_{G/K}(\lambda)^K$. The $K$-invariant space in $V(\lambda)$ is one dimensional, thus the spherical functions $\{f_\lambda\}_{\lambda \in \Lambda_{G/K}^+}$ form a basis for the space of $K$-invariant functions
$$\CC[G]^{K\times K} = \bigoplus E_{G/K}(\lambda)^K.$$

More explicitly, if $\lambda \in \Lambda_{G/K}^+$, the spherical function $f_\lambda$ can be constructed (up to a scalar factor) as a matrix coefficient by taking non-zero $K$-invariants vectors $v_\lambda \in V(\lambda)$ and $\psi_\lambda \in V(\lambda)^*$ and setting
$$
	f_\lambda(gK) = \langle \psi_\lambda,g.v_\lambda \rangle.
$$

The decomposition of the product of irreducible $G$-submodules inside $\CC[G/K]$ is encoded in the structure constants for the multiplication of the corresponding spherical functions (see {\cite[Theorem 3.2]{Ruit}}).

\begin{proposition} 
If $\lambda, \mu, \nu \in \Lambda_{G/K}^+$ and if $f_\lambda f_\mu = \sum a_{\lambda,\mu}^\nu  f_\nu$, then
$$
E_{G/K}(\nu) \subset E_{G/K}(\lambda) \cdot E_{G/K}(\mu) \Longleftrightarrow a_{\lambda,\mu}^\nu \neq 0.
$$
\end{proposition}

More generally, a similar characterization can be given in terms of the support of the tensor product of the $K$-semi-invariant vectors of the quasi-spherical representations (see {\cite[Lemma 19]{CLM}, \cite[Lemma 1.2]{BGM}}).

\begin{proposition}\label{prop:supporto_invarianti}
Let $\lambda, \mu, \nu \in \Omega_{G/K}^+$ and let $\chi, \chi' \in \calX(K)$ be such that the semi-invariant subspaces $V(\lambda)^{(K)}_\chi$ and $V(\mu)^{(K)}_{\chi'}$ are non-zero. Let $v_\lambda \in V(\lambda)^{(K)}_\chi$ and $v_\mu \in V(\mu)^{(K)}_{\chi'}$ be non-zero, and let $\pi_\nu : V(\lambda) \otimes V(\mu) \rightarrow V[\nu]$ denote the projection onto the isotypic component $V[\nu] \subset V(\lambda) \otimes V(\mu)$ of highest weight $\nu$. Then
$$
E_{G/K}(\nu,\chi+\chi') \subset E_{G/K}(\lambda,\chi) \cdot E_{G/K}(\mu,\chi') \Longleftrightarrow \pi_\nu(v_\lambda \otimes v_\mu)\neq 0.
$$
\end{proposition}

%%%%%%%%%%%%%%%%%%%%%%%%%%%%%%%%%%%%%%%
%%%%%%%%%%%%%%%%%%%%%%%%%%%%%%%%%%%%%%%
\section{Multiplication of spherical functions of symmetric pairs of type A}\label{sec:sym}
%%%%%%%%%%%%%%%%%%%%%%%%%%%%%%%%%%%%%%%
%%%%%%%%%%%%%%%%%%%%%%%%%%%%%%%%%%%%%%%

Let $G$ be semisimple and simply connected, and let $\vartheta: G \rightarrow G$ be an algebraic involution. We denote by $K = G^\vartheta$ the set of fixed points. Then $K$ is a connected reductive subgroup of $G$, which has finite index in $\rmN_G(K)$. 

Let $A \subset G$ be a maximal split torus, that is $\vartheta(a) = a^{-1}$ for all $a \in A$ and $A$ is maximal with this property. Fix a maximal torus $T$ containing $A$, then $\vartheta(T) = T$. Let us denote by the same letter $\vartheta$ the induced involution on $\mathfrak t^*$, then $\vartheta(\Phi) = \Phi$ and the Killing form is preserved by $\vartheta$. We also choose the set of positive roots and the Borel subgroup $B \supset T$ in such a way that $\vartheta(\alpha) \in \Phi^-$ whenever $\alpha \in \Phi^+$ and $\alpha \neq \vartheta(\alpha)$ (see \cite[Lemma~1.2]{CP}).

Denote $X = G/K$, and let $A_X \simeq A/A\cap K$ be the image of $A$ in $X$. Notice that
$$
	A \cap K = \{a \in A \; | \; a^2 = 1\}
$$
thus $\calX(A_X) = 2\calX(A)$ is identified with the lattice 
\[
	\Lambda_X = \{\chi - \vartheta(\chi) \; | \; \chi \in \calX(T)\} .
\]

Define 
$$
	\widetilde \Phi_X = \{\alpha - \vartheta(\alpha) \; | \; \alpha \in \Phi, \alpha \neq \vartheta(\alpha)\}
$$
the set of restricted roots of $X$. Then $\widetilde \Phi_X$ is a (possibily non-reduced) root system in $\Lambda_X \otimes \RR \simeq \calX(A) \otimes \RR$, with base
$$
	\widetilde \Delta_X = \{\alpha - \vartheta(\alpha) \; | \; \alpha \in \Delta, \ \vartheta(\alpha) \neq \alpha \}
$$
and with Weyl group
$$W_X \simeq \rmN_G(A)/\rmZ_G(A) \simeq \rmN_K(A)/\rmZ_K(A).$$

The weight lattice and the root lattice of $\widetilde\Phi_X$ are respectively identified with $\Lambda_X \simeq \calX(A_X)$ and with $\calX(A/\rmN_A(K)) \subset \calX(A_X)$ (see \cite[Lemmas~2.3 and 3.1]{Vu2}). Finally, the monoid of the spherical weights $\Lambda_X^+$ is obtained by intersecting the monoid of the dominant weights with the sublattice of $\calX(T)$ defined by $\calX(A_X)$ (see \cite[Th\'eor\`eme~3]{Vu1}):
$$
\Lambda_X^+ = \Lambda_X \cap \calX(T)^+.
$$

By a result of Richardson \cite[Corollary 11.5]{Ric}, the restriction of functions yields an algebra isomorphism
$$
	\CC[X]^K \stackrel{\sim}{\longrightarrow} \CC[A_X]^{W_X}.
$$

Up to normalization, the restriction of the spherical function $f_\lambda$ to $A_X$ is a specialization of the Jacobi polynomial $P_\lambda^{(k)}$ associated with the root system $\widetilde \Phi_X$ (see \cite[Section~8.4]{HO}). More precisely, let $m = (m_{\widetilde \alpha})$ be the multiplicity function of the restricted roots, that is, for $\widetilde \alpha \in \widetilde \Phi_X$,
$$
	m_{\widetilde \alpha} = | \{\beta \in \Phi \; | \; \beta - \vartheta(\beta) = \widetilde \alpha\}|.
$$
Then up to a scalar factor, the restriction ${f_\lambda}_{|A_X}$ coincides with the specialized Jacobi polynomial $P_\lambda^{(m/2)}$.

Therefore, the decomposition of the products of the irreducible $G$-modules inside $\CC[X]$ is encoded in the structure constants for the multiplication of the corresponding specialized Jacobi polynomials:

If $\lambda, \mu, \nu \in \Lambda_X^+$ and if $P_\lambda^{(m/2)} \, P_\mu^{(m/2)} = \, \sum a_{\lambda,\mu}^\nu  P_\nu^{(m/2)}$, then
$$
E_X(\nu) \subset E_X(\lambda) \cdot E_X(\mu) \Longleftrightarrow a_{\lambda,\mu}^\nu \neq 0.
$$

Let $\alpha \in \Phi$ be such that $\alpha\neq \vartheta(\alpha)$: then either $\vartheta(\alpha) = -\alpha$, or $\alpha$ and $\vartheta(\alpha)$ are orthogonal, or $\langle \vartheta(\alpha), \alpha^\vee \rangle = 1$ (in which case $\alpha - \vartheta(\alpha) \in \Phi^+$), see \cite[Lemma~2.3]{Vu2}. Thus if $\sigma \in \widetilde \Delta_X$, say $\sigma = \alpha - \vartheta(\alpha)$, we define the coroot $\sigma^\vee \in \Lambda_X^\vee$ as follows:
\begin{itemize}
	\item[i)] If $\vartheta(\alpha) = -\alpha$, then 
$$
	\sigma^\vee = \tfrac{1}{2}{\alpha^\vee}_{|\Lambda_X}
$$

	\item[ii)] If $\alpha$ and $\vartheta(\alpha)$ are orthogonal, then
$$
	\sigma^\vee = {\alpha^\vee}_{|\Lambda_X}
$$

	\item[iii)] If $\alpha - \vartheta(\alpha) \in \Phi^+$, then
$$
	\sigma^\vee = {(\alpha- \vartheta(\alpha))^\vee}_{|\Lambda_X}
$$
\end{itemize}
Notice that the previous definition only depends on $\sigma$, and not on $\alpha$. If indeed $\lambda \in \Lambda_X$, then in all cases we have
$$
	\langle \lambda, \sigma^\vee\rangle = \frac{2(\lambda, \alpha-\vartheta(\alpha))}{||\alpha-\vartheta(\alpha)||^2}
$$

For later use, we now introduce a slight modification of the restricted root system. Denote
$$
	\Delta_X^\dag=\{\sigma\in\widetilde \Delta_X \cap 2\Delta \;|\; \langle\sigma,\tau^\vee\rangle\in2\mathbb Z\ \forall \tau\in\widetilde\Delta_X\},
$$
and define
$$
	\Delta_X = (\widetilde \Delta_X \setminus \Delta_X^\dag) \cup \tfrac{1}{2}\Delta_X^\dag.
$$
Then $\Delta_X$ is the base of a (reduced) root system $\Phi_X$ (see e.g. \cite[Proposition~3.1]{KS}): indeed, the scalar product between any two distinct elements in $\Delta_X$ remains non-positive, and the pairing among roots remains integer. Notice that $\Phi_X$ has also Weyl group $W_X$. If moreover $\Delta_X$ is irreducible of rank $>1$ and distinct from $\widetilde \Delta_X$, then $\Delta_X$ and $\widetilde \Delta_X$ have necessarily dual types.

Accordingly we extend the weight lattice with the normalized restricted roots, and we define $\Xi_X$ to be the lattice generated by $\Lambda_X$ and $\Delta_X$ (namely by $\Lambda_X$ and $\tfrac12\Delta_X^\dag$).

For $\sigma \in \Delta_X$, we define the coroot $\sigma^\vee \in \Xi_X^\vee$ similarly as we did before: if $\sigma \in \Phi^+$ then $\sigma^\vee \in \Xi_X^\vee$ is the restriction of the associated coroot in $\Lambda^\vee$, if instead $\sigma=2\alpha$ with $\alpha\in\Delta$ then we define $\sigma^\vee=\tfrac12{\alpha^\vee}_{|\Xi_X}$, and finally if $\sigma = \alpha - \vartheta(\alpha)$ with $\langle \vartheta(\alpha), \alpha^\vee \rangle = 0$ then we define $\sigma^\vee = {\alpha^\vee}_{|\Xi_X}$. 

Set 
\[\Delta_X^\vee=\{\sigma^\vee\ :\ \sigma\in\Delta_X\},\]
regarded as a subset of the dual lattice $\Xi_X^\vee$.
Then the triple $\calR_X = (\Xi_X, \Delta_X, \Delta_X^\vee)$ is a based root datum, in the following (usual) sense:

A \textit{based root datum} is a triple $(\Xi,S,\alpha\mapsto\alpha^\vee)$ such that
\begin{itemize}
\item[i)] $\Xi$ is a free abelian group of finite rank, $S$ is a finite subset of $\Xi$ and $\alpha\mapsto\alpha^\vee$ is a map from $S$ into $\Xi^\vee$,
\item[ii)] for all $\alpha\in S$ the pairing $\langle\alpha,\alpha^\vee\rangle$ is equal to 2,
\item[iii)] the group $W$ of automorphisms of $\Xi$ generated by the reflections
\[
s_\alpha\colon \Xi\longrightarrow\Xi \qquad \qquad
\xi \longmapsto \xi-\langle x,\alpha^\vee\rangle\alpha\] 
with $\alpha\in S$ is finite,
\item[iv)] the $W$-orbit of $S$ is a reduced root system $R$ in the real subspace of $\Xi_\RR$ generated by $S$, and $S$ is a base of $R$.
\end{itemize} 

Notice that in our case the root datum $\calR_X$ is \emph{semisimple}, namely the base $S$ spans $\Xi_\RR$ as a vector space.

Let 
\[\Xi_X^+:=\{\lambda\in\Xi_X\ :\ \langle\lambda,\sigma^\vee\rangle\geq0\ \forall\ \sigma\in \Delta_X\},\] 
be the monoid of the dominant weights of $\calR_X$. Then we have a (possibly strict) inclusion $\Lambda_X^+ \subset \Xi_X^+$, and  
$\Xi_X^+ = \Xi_X \cap \Lambda^+$.

On the weight lattice $\Xi_X$, we will consider the partial order $\leq_X$ defined by $\widetilde \Delta_X$ as follows
$$
	\pi_1 \leq_X \pi_2 \quad\mbox{if}\quad \pi_2 - \pi_1 \in \NN \widetilde \Delta_X
$$
When $\widetilde \Delta_X = \Delta_X$ this is the usual dominance order on $\Xi_X$, otherwise it is a refinement of the dominance order where the simple roots in $\tfrac12\Delta_X^\dag$ are only taken with even coefficients.

%%%%%%%%%%%%%%%%%%%%%%%%%%%%%%%%%%%%%%%
%%%%%%%%%%%%%%%%%%%%%%%%%%%%%%%%%%%%%%%
\subsection{A conjectural rule when $\Phi_X$ is of type $\sfA$.}
%%%%%%%%%%%%%%%%%%%%%%%%%%%%%%%%%%%%%%%
%%%%%%%%%%%%%%%%%%%%%%%%%%%%%%%%%%%%%%%

When $\widetilde \Phi_X$ (or equivalently $\Phi_X$) is of type $\sfA$, then up to a normalization the Jacobi polynomials $P_\lambda^{(k)}$ coincide with the Jack symmetric polynomials $J_\lambda^{(1/k)}$ (see e.g. \cite[Proposition~3.3]{BO}). Thus in this case the restriction of the spherical function $f_\lambda$ to $A_X$ is a scalar multiple of the specialized Jack polynomial $J_\lambda^{(2/m)}$. 

As is well known, when dealing with symmetric polynomials, it is often convenient to consider symmetric functions in infinitely many variables. Here we do the same, and instead of working with the Jack polynomials we pass to the Jack symmetric functions $J_\lambda^{(k)}$, which form a basis, when $\lambda$ vary among all possible partitions, of the ring of symmetric functions in infinitely many variables with coefficients in $\mathbb Q(k)$. As a general reference on the Jack symmetric functions, see \cite{St} and \cite[Section VI.10]{Mac}.

We now recall a conjecture of Stanley concerning the multiplication of the Jack symmetric functions.

\begin{conjecture}[{\cite[Conjecture~8.3]{St}}] \label{conj:stanley}
Write $J_\lambda^{(k)} J_\mu^{(k)} = \sum_\nu \frac{g_{\lambda,\mu}^\nu(k)}{j_\nu(k)} J_\nu^{(k)}$. Then the function $g_{\lambda,\mu}^\nu(k)$ is a polynomial in $k$, with non-negative integer coefficients.
\end{conjecture}

To explain the denominator $j_\nu(k)$, recall that the Jack symmetric functions are pairwise orthogonal with respect to the scalar product $\langle\ ,\ \rangle_k$ defined on the basis of the power-sum symmetric functions $p_\lambda$ ($\lambda$ being a partition) as follows:
$$
\langle p_\lambda, p_\mu \rangle_k = \left\{
\begin{array}{cc}
z_\lambda \ k^{\ell(\lambda)} & \text{if $\lambda =\mu$} \\
0 & \text{if $\lambda \neq \mu$} 
\end{array} \right.
$$
Here $\ell(\lambda)$ denotes the length of $\lambda$, and if $m_i$ is the number of parts in $\lambda$ equal to $i$ then
$$z_\lambda = (1^{m_1} 2^{m_2} \cdots) m_1! \ m_2! \cdots $$
Therefore, expanding the product of two Jack symmetric functions in terms of Jack symmetric functions, one has
\[J_\lambda^{(k)} J_\mu^{(k)} = \sum_\nu \frac{\langle J_\lambda^{(k)} J_\mu^{(k)}, J_\nu^{(k)}\rangle_k}{\langle J_\nu^{(k)},J_\nu^{(k)}\rangle_k}\ J_\nu^{(k)} .\]
Then $j_\nu(k) = \langle J_\nu^{(k)},J_\nu^{(k)}\rangle_k$, and $g_{\lambda,\mu}^\nu(k)=\langle J_\lambda^{(k)} J_\mu^{(k)}, J_\nu^{(k)}\rangle_k$. There is a well known combinatorial formula which expresses $j_\nu(k)$ as a polynomial in $k$ with non-negative integer coefficients in terms of hooks of $\nu$ (see \cite[Theorem~5.8]{St}). 

It is known that the functions $g_{\lambda,\mu}^\nu(k)$ are polynomials in $k$ with integer coefficients: this follows from a result of Knop and Sahi \cite{KS}, showing that the coefficients of the expansion of the Jack symmetric functions in terms of the monomial symmetric functions are polynomial in $k$ with non-negative integer coefficients.

Conjecture \ref{conj:stanley} is known to be true when one of the two partitions $\lambda$ and $\mu$ consists of a single row: this follows from a generalization of the Pieri rule for Jack symmetric functions proved by Stanley himself \cite[Theorem~6.1]{St}.

Let us come back to a symmetric space $X=G/K$ with restricted root system of type $\sfA$. If we normalize properly the spherical functions $f_\lambda$ and write $f_\lambda f_\mu = \sum a_{\lambda,\mu}^\nu \, f_\nu$, then $a_{\lambda,\mu}^\nu = g_{\lambda,\mu}^\nu(2/m)$. On the other hand, by the very definition of the scalar product $\langle \ ,\ \rangle_k$, specializing the Jack symmetric function $J_\lambda^{(k)}$ at $k = 1$ yields a non-zero scalar multiple of the Schur function $s_\lambda$. Thus $g_{\lambda,\mu}^\nu(1)$ is non-zero if and only if the Littlewood-Richardson coefficient $c_{\lambda,\mu}^\nu$ associated to $\lambda,\mu,\nu$ in the restricted root system associated to $G/K$ is non-zero.

For a symmetric space $X$ with restricted root system of type $\sfA$, we get then
\begin{equation}	\label{g}
E_X(\nu) \subset E_X(\lambda) \cdot E_X(\mu) \Longleftrightarrow g_{\lambda, \mu}^\nu(2/m) \neq 0 ,
\end{equation}
and Stanley's conjecture immediately implies
\begin{equation}	\label{stanley-cor}
E_X(\nu) \subset E_X(\lambda) \cdot E_X(\mu) \Longleftrightarrow c_{\lambda,\mu}^\nu \neq 0.
\end{equation}

When the rank of the restricted root system is one, only one-row partitions occur as dominant weights, therefore Equation \eqref{stanley-cor} holds true thanks to Stanley's generalization of the Pieri rule mentioned above.

To be more explicit, let $G_X$ be the connected semisimple group defined by the based root datum $\calR_X = (\Xi_X, \Delta_X, \Delta_X^\vee)$, and correspondingly let $T_X \subset G_X$ and $B_X \subset G_X$ be a maximal torus and a Borel subgroup containing $T_X$ inducing $\calR_X$. For $\pi \in \Xi_X^+$ (and in particular for $\pi \in \Lambda_X^+$) let $V_X(\pi)$ be the irreducible $G_X$-module with highest weight $\pi$. Then Stanley's Conjecture \ref{conj:stanley} yields the following (weaker) conjecture.

\begin{conjecture}	\label{conj:stanley_cor}
Let $X$ be a symmetric variety with restricted root system of type $\sfA$ and let $\lambda, \mu, \nu \in \Lambda_X^+$. Then
$$
E_X(\nu) \subset E_X(\lambda) \cdot E_X(\mu) \Longleftrightarrow
\left\{ \begin{array}{c}
V_X(\nu) \subset V_X(\lambda) \otimes V_X(\mu)  \phantom{\Big|} \\
\nu \leq_X \lambda + \mu
\end{array} \right.
$$
\end{conjecture}

When the restricted root system of $X$ is of type $\sfA_1$,  the previous conjecture holds true by Stanley's Pieri rule for Jack symmetric functions.

Notice that Conjecture \ref{conj:stanley_cor} is easily shown when $G = H \times H$ and $K = \diag(H)$ for some reductive algebraic group $H$ (even in types other than $\sfA$), in which case $X \simeq H$ regarded as $H \times H$-variety (see e.g. \cite[Lemma 3.4]{DC}). 

\begin{remark}
Notice that in all cases of Table \ref{tab:sym_pairs_A} but case Sym.A1 with $n=2$, it always holds the equality $\Delta_X = \widetilde \Delta_X$. In particular in all these cases Conjecture \ref{conj:stanley_cor} follows immediately from \eqref{stanley-cor}, and the condition $\nu\leq_X \lambda+\mu$ is redundant. On the other hand, to deduce Conjecture \ref{conj:stanley_cor} from \eqref{stanley-cor} in the case $X = \SL(2)/T$ we need to use the saturation property of the tensor product of $\SL(2)$.
\end{remark}

\begin{remark}		\label{oss:root_data}
Rather than $\mathcal R_X$, in Conjecture \ref{conj:stanley_cor} it would have been more natural to use the based root datum $\widetilde \calR_X = (\Lambda_X, \widetilde \Delta_X, \widetilde \Delta_X^\vee)$, and define the semisimple group $G_X$ accordingly. Indeed this formulation follows immediately from \eqref{stanley-cor}, and is also simpler because the inequality $\nu\leq_X \lambda+\mu$ becomes redundant. However computational experiments suggest that the formulation that we adopted is closer to possible generalizations (see Section \ref{sec:Levi}, where we consider the case of the symmetric subgroups of Hermitian type).
\end{remark}

In a slightly different form, Conjecture \ref{conj:stanley_cor} was also considered by Graham and Hunziker \cite{GH}. 

For completeness, we list in Table \ref{tab:sym_pairs_A} all the symmetric varieties with a restricted root system of type $\sfA$. In all cases, we also give the simple restricted roots and a minimal set of generators for the weight monoid $\Lambda^+_X$.  

\begin{table}[H]
\caption{Symmetric pairs with restricted root system of type $\sfA_r$} \label{tab:sym_pairs_A}
\begin{tabular}{|c|c|c|c|c|}
\hline 
& $\phantom{\Big|}$ $G$ & $K$ & $\Phi_{G/K}$ & $m_\sigma$ \\
    \hline    \hline
Sym.A1 & $\phantom{\Big|}$ $\SL(n)$, $n\geq2$ & $\SO(n)$ & $\sfA_{n-1}$ & 1 \\
    \hline
Sym.A2 & $\phantom{\Big|}$ $\SL(n) \times \SL(n)$, $n\geq2$ & $\diag(\SL(n))$ & $\sfA_{n-1}$ & 2 \\
    \hline
Sym.A3 & $\phantom{\Big|}$ $\SL(2n)$, $n\geq2$ & $\Sp(2n)$ & $\sfA_{n-1}$ & 4 \\
    \hline 
Sym.A4 & $\phantom{\Big|}$ $\Spin(n)$, $n\geq5$ & $\Spin(n-1)$ & $\sfA_1$ & $n-2$ \\
    \hline
Sym.A5 & $\phantom{\Big|}$ $\sfE_6$ & $\sfF_4$ & $\sfA_2$ & 8 \\
    \hline 
\end{tabular}
\end{table}

\begin{itemize}
	\item[i)] $\mathbf{SL(2)/SO(2)}$: $\Delta_X = \{\alpha_1\}$, $\Lambda^+_X$ is generated by $2 \omega_1$;
	\item[ii)] $\mathbf{SL(n)/SO(n)}$, $n\geq3$: $\Delta_X = \{2\alpha_i\ : \ 1\leq i\leq n-1\}$, $\Lambda^+_X$ is generated by the weights $2 \omega_i$ with $ 1\leq i\leq n-1$;
	\item[iii)] $\big(\mathbf{SL(n) \times SL(n)\big)/diag(SL(n))}$, $n \geq 2$: $\Delta_X = \{\alpha_i + \alpha_i'\ :\ 1\leq i\leq n-1\}$, $\Lambda^+_X$ is generated by the weights $\omega_i + \omega_i'$ with $ 1\leq i\leq n-1$;
	\item[iv)] $\mathbf{SL(2n)/Sp(2n)}$, $n \geq 2$: $\Delta_X = \{\alpha_{2i-1} + 2\alpha_{2i}+\alpha_{2i+1}\ :\ 1\leq i\leq n-1\}$, $\Lambda^+_X$ is generated by the weights $\omega_{2i}$ with $ 1\leq i\leq n-1$;
	\item[v)] $\mathbf{Spin(2n+1)/Spin(2n)}$, $n \geq 2$: $\Delta_X = \{2(\sum_{i=1}^n\alpha_i)\}$, $\Lambda^+_X$ is generated by $\omega_1$;
	\item[vi)] $\mathbf{Spin(2n)/Spin(2n-1)}$, $n \geq 3$: $\Delta_X = \{2(\sum_{i=1}^{n-2}\alpha_i) + \alpha_{n-1} + \alpha_n\}$, $\Lambda^+_X$ is generated by $\omega_1$;
	\item[vii)] $\mathbf{E_6/F_4}$: $\Delta_X = \{2\alpha_1 + \alpha_2 + 2\alpha_3 + 2\alpha_4 + \alpha_5, \; \alpha_2 + \alpha_3 + 2\alpha_4 + 2\alpha_5 + 2\alpha_6\}$, $\Lambda^+_X$ is generated by $\omega_1$ and $\omega_6$.
\end{itemize}

%We will state a similar conjecture also for the symmetric varieties of Hermitian type , and then for the non-reduced symmetric spherical varieties (again, with one exception).
%
%Before we consider some natural generalizations of the restricted root system of a symmetric variety.

%%%%%%%%%%%%%%%%%%%%%%%%%%%%%%%%%%%%%%%
%%%%%%%%%%%%%%%%%%%%%%%%%%%%%%%%%%%%%%%
\section{Root systems associated to a spherical homogeneous space}\label{sec:sph}
%%%%%%%%%%%%%%%%%%%%%%%%%%%%%%%%%%%%%%%
%%%%%%%%%%%%%%%%%%%%%%%%%%%%%%%%%%%%%%%

We now recall how to attach a based root system (or more precisely, three based root systems) to a spherical homogeneous space $G/K$, which for the purposes of the present paper we can assume affine. General references for these constructions are \cite{Br} and \cite{Kn}.

Let $\Lambda_{G/K}$ be the weight lattice of $G/K$, defined as
$$
	\Lambda_{G/K} = \{\chi \in \calX(T) \; | \; \CC(G/K)^{(B)}_\chi \neq 0\}.
$$ 
Since $K$ is reductive, notice that $\Lambda_{G/K}$ is generated as a lattice by $\Lambda^+_{G/K}$. In all the cases that we will consider, we will actually have $\Lambda^+_{G/K} = \Lambda_{G/K} \cap \Lambda^+$. 

%Then $\Lambda^+_{G/K} = \Lambda_{G/K} \cap \Lambda^+$ is the monoid of the $K$-spherical weights. Notice that $\Lambda_{G/K}$ is generated as a lattice by $\Lambda^+_{G/K}$, since $K$ is reductive.

%Following \cite{BP} and \cite{Kn}, we can define a root system by looking at the multiplication of spherical modules inside the invariant space $\CC[G]^K$ or the semi-invariant space $\CC[G]^{(K)}$., essentially go back to Brion and Pauer .

Looking at the multiplication of spherical modules, we can associate three root systems to $G/K$. For a fixed $G/K$, the three root systems will have the same Weyl group $W_{G/K}$: what changes is only the root normalization.\\

$\bullet$ \textbf{The $\rmn$-spherical roots of $G/K$.}
The first root system, that we denote by $\Phi^\rmn_{G/K}$, is defined by considering the multiplication of spherical modules inside the invariant space $\CC[G]^K$. If $\lambda, \mu, \nu \in \Lambda^+_{G/K}$ is a triple of spherical weights such that $E(\nu) \subset E(\lambda) \cdot E(\mu)$, then $\lambda+\mu-\nu \in \Lambda_{G/K} \cap \NN\Delta$. Denote by $\calM^\rmn_{G/K}$ the monoid generated by all possible differences $\lambda+\mu-\nu \in \Lambda_{G/K}$ with $\lambda, \mu, \nu \in \Lambda^+_{G/K}$ as above. This is a free monoid (see \cite[Theorem 1.3]{Kn} and \cite[Theorem 4.11]{ACF}). Its (unique) set of free generators, denoted by $\Delta^\rmn_{G/K}$, is the set of the \textit{$\rmn$-spherical roots of $G/K$}. It is the base of a reduced root system, $\Phi^\rmn_{G/K}$ (see \cite[Theorem 1.3]{Kn}).\\

$\bullet$ \textbf{The $\rmsc$-spherical roots of $G/K$.} The second root system, that we denote by $\Phi^\rmsc_{G/K}$, is defined by considering the multiplication of quasi-spherical modules inside the semi-invariant space $\CC[G]^{(K)}$. If $\lambda, \mu, \nu \in \Omega^+_{G/K}$ is a triple of quasi-spherical weights such that $E_{\chi+\chi'}(\nu) \subset E_\chi(\lambda) \cdot E_{\chi'}(\mu)$, then $\lambda+\mu-\nu \in \Lambda_{G/K} \cap \NN \Delta$. Let $\calM^\rmsc_{G/K}$ be the monoid generated by all possible differences $\lambda+\mu-\nu$ with $\lambda, \mu, \nu \in \Omega^+_{G/K}$ as above. Again, this is a free monoid (see \cite[Corollary 7.6]{Kn} and \cite[Proposition 5]{BGM}). Its (unique) set of free generators, denoted by $\Delta^\rmsc_{G/K}$, is the set of the \textit{$\rmsc$-spherical roots of $G/K$}. It is the base of a reduced root system, $\Phi^\rmsc_{G/K}$.\\

$\bullet$ \textbf{The minimal spherical roots of $G/K$.}
By definition we have inclusions $\calM^\rmn_{G/K} \subset \calM^\rmsc_{G/K} \subset \Lambda_{G/K}$.
The two root monoids $\calM^\rmn_{G/K}$ and $\calM^\rmsc_{G/K}$ actually generate the same cone in $\Lambda_{G/K} \otimes_\ZZ \QQ$. However neither the $\rmn$-spherical roots nor the $\rmsc$-spherical roots need to be primitive elements inside $\Lambda_{G/K}$: the set of the \textit{minimal spherical roots of $G/K$}, denoted by $\Delta^{\min}_{G/K}$, is the set of the primitive elements in $\Lambda_{G/K}$ associated to the extremal rays of the cone generated by $\calM^\rmn_{G/K}$ in $\Lambda_{G/K} \otimes_\ZZ \QQ$. The minimal spherical roots of $G/K$ form the base of a reduced root system, that we denote by $\Phi^{\min}_{G/K}$.\\

It is clear from their construction that the three root systems $\Phi^{\min}_{G/K}$, $\Phi^\rmsc_{G/K}$, $\Phi^\rmn_{G/K}$ all have the same Weyl group $W_{G/K}$.

If $G/K$ is a symmetric variety, we have $\widetilde \Delta_{G/K} = \Delta^\rmn_{G/K}$ (see \cite[Theorem~6.7]{Kn}).

Both the $\rmn$-spherical roots and the $\rmsc$-spherical roots of $G/K$ arise as minimal spherical roots of a quotient of $G/K$. Denote indeed by $\bar K$ the spherical closure of $K$, defined as the kernel of the action of $\rmN_G(K)$ on $\calX(K)$. Then both $\rmN_G(K)$ and $\bar K$ are spherical subgroups of $G$, which satisfy
$$
	\Delta^\rmn_{G/K} = \Delta^{\min}_{G/\rmN_G(K)}
	\qquad \qquad 
	\Delta^\rmsc_{G/K} = \Delta^{\min}_{G/\overline K}.
$$
This explains our terminology. Moreover, we have equalities
$$
\Lambda_{G/\rmN_G(K)} = \ZZ \Delta^\rmn_{G/K}
\qquad \qquad 
\Lambda_{G/\bar K} = \ZZ \Delta^\rmsc_{G/K}.
$$
The inclusion $\ZZ \Delta^{\min}_{G/K} \subset \Lambda_{G/K}$ is strict in general: it is an equality precisely when $G/K$ admits a wonderful completion (which always happens for $G/\bar K$ and $G/\rmN_G(K)$). Notice that in general the minimal spherical roots of $G/K$ do not necessarily belong to the root lattice, differently from the $\rmn$-spherical roots and from the $\rmsc$-spherical roots which by construction are always in the root lattice of $G$.

\begin{remark}	\label{oss:norm-reticoli}
The weight lattices of $G/K$, $G/\bar K$ and $G/\rmN_G(K)$ allow easily to compare $K$, $\bar K$ and $\rmN_G(K)$: indeed the quotient $\rmN_G(K)/ K$ is diagonalizable (see \cite[Corollaire 5.2]{BP}) and we have inclusions
$$
	\Lambda_{G/\rmN_G(K)} \subset \Lambda_{G/\bar K} \subset \Lambda_{G/K},
$$
inducing isomorphisms (see e.g.\ \cite[Lemma 2.4]{Ga})
\begin{gather*}
	\calX(\rmN_G(K)/K) \simeq \Lambda_{G/K}/\Lambda_{G/\rmN_G(K)},\\
	\calX\big(\rmN_G(K)/\bar K \big) \simeq \Lambda_{G/\bar K}/\Lambda_{G/\rmN_G(K)},
\qquad \qquad 	\calX\big(\bar K/K \big) \simeq \Lambda_{G/K}/\Lambda_{G/\bar K}.
\end{gather*}
\end{remark}

%Notice that in all cases of Table \ref{tab:sph_pairs_A} it holds
%$$
%	\Delta^\rmsc_{G/K} = (\Delta^\rmn_{G/K} \setminus 2\grD) \cup (\tfrac{1}{2} \Delta^\rmn_{G/K} \cap \grD) 
%$$

As already anticipated in the case of a symmetric variety, we define a fourth normalization for the spherical roots, setting
$$
	\Delta_{G/K} = (\Delta^\rmn_{G/K} \setminus \Delta_{G/K}^\dag) \cup \tfrac{1}{2}\Delta_{G/K}^\dag
$$
where
$$
	\Delta_{G/K}^\dag=\{\sigma\in\Delta^\rmn_{G/K} \cap 2\Delta \;|\; \langle\sigma,\tau^\vee\rangle\in2\mathbb Z\ \; \forall \tau\in\Delta^\rmn_{G/K}\}.
$$

Since $\Delta^\rmn_{G/K}$ is obtained from $\Delta^\rmsc_{G/K}$ by doubling some particular elements in $\Delta^\rmsc_{G/K} \cap \Delta$ (see \cite[Section~2.4]{BL}), it turns out that we also have
$$
	\Delta_{G/K} = (\Delta^\rmsc_{G/K} \setminus \Delta_{G/K}^\dag) \cup \tfrac{1}{2}\Delta_{G/K}^\dag .
$$
We denote by $\Phi_{G/K}$ the (reduced) root system generated by $\Delta_{G/K}$. Notice that $\Phi_{G/K}$ has also Weyl group $W_{G/K}$. In all cases that we will be interested (see Table~\ref{tab:sph_pairs_A}) we have indeed $\Delta_{G/K} =  \Delta^\rmsc_{G/K}$: this is always true when $K$ is connected.

If $\sigma \in \Delta^\rmn_{G/K}$, then the following three possibilities occur (see \cite[Th\'eor\`eme~2.6]{Br}): either $\sigma \in \Phi^+$ is a positive root, or $\sigma \in 2 \Delta$ is the double of a simple root, or $\sigma = \alpha + \beta$ is the sum of two strongly orthogonal positive roots (that is, neither $\alpha + \beta$ nor $\alpha-\beta$ is in $\Phi$). Since $\Delta_{G/K}$ and $\Delta^\rmsc_{G/K}$ are obtained from $\Delta^\rmn_{G/K}$ by replacing some elements in $2\Delta \cap \Delta^\rmn_{G/K}$ with the corresponding simple roots, the same description applies to $\Delta^\rmsc_{G/K}$ and $\Delta_{G/K}$.

%Furthermore (see \cite[Section~2]{Lu1}), in the latter case, it can be the double of a simple root, i.e.\ $\sigma=2\alpha$ with $\alpha\in \Delta$, or the sum of two strongly orthogonal roots, i.e.\  $\sigma=\beta+\gamma$ with $\beta,\gamma\in\Phi^+$, $\langle\beta,\gamma^\vee\rangle=0$ and $\beta+\gamma\not\in\Phi^+$. 

There is actually a finite list of possible cases: the shape of a $\rmn$-spherical root, written as sum of simple roots, is necessarily one of those reported in Table~\ref{tab:sc_sph_roots} (where we also specify the type of the root subsystem of $\Phi$ generated by the support of the spherical root).

\begin{table}[H]
\caption{$\rmsc$-spherical roots} \label{tab:sc_sph_roots}
\begin{tabular}{lll}
$\sigma$ & & type of $\sigma$ \\
\hline
$\alpha$ & & $\sfA_1$ \\
$2\alpha$ & & $\sfA_1$ \\
$\alpha+\alpha'$ & & $\sfA_1\times\sfA_1$ \\
$\alpha_1+\ldots+\alpha_m$ & & $\sfA_m$, $m\geq2$ \\
$\alpha_1+2\alpha_2+\alpha_3$ & & $\sfA_3$ \\
$\alpha_1+\ldots+\alpha_m$ & & $\sfB_m$, $m\geq2$ \\
$2(\alpha_1+\ldots+\alpha_m)$ & & $\sfB_m$, $m\geq2$ \\
$\alpha_1+2\alpha_2+3\alpha_3$ & & $\sfB_3$ \\
$\alpha_1+2(\alpha_2+\ldots+\alpha_{m-1})+\alpha_m$ & & $\sfC_m$, $m\geq3$ \\
$2(\alpha_1+\ldots+\alpha_{m-2})+\alpha_{m-1}+\alpha_m$ & & $\sfD_m$, $m\geq4$ \\
$\alpha_1+2\alpha_2+3\alpha_3+2\alpha_4$ & & $\sfF_4$ \\
$\alpha_1+\alpha_2$ & & $\sfG_2$ \\
$4\alpha_1+2\alpha_2$ & & $\sfG_2$
\end{tabular}
\end{table}

In general $\Delta_{G/K}$ is not contained in $\Lambda_{G/K}$. Thus we extend accordingly the weight lattice, and define $\Xi_{G/K}$ to be the lattice generated by $\Lambda_{G/K}$ together with $\Delta_{G/K}$, namely by $\Lambda_{G/K}$ and $\tfrac{1}{2}\Delta_{G/K}^\dag$.

The following lemma will allow us to define the spherical coroots in $\Xi_{G/K}^\vee$, in analogy to what we did in the symmetric case.

\begin{lemma}	\label{lemma:coradici sferiche}
Let $G/K$ be a spherical variety and let $\sigma \in \Delta^\rmsc_{G/K} \setminus \Phi^+$.

\begin{itemize}
\item[i)] 	If $\sigma = 2\alpha$ with $\alpha \in \Delta$, then $\langle \lambda, \alpha^\vee \rangle \in 2 \ZZ$ for all $\lambda \in \Xi_{G/K}$.

\item[ii)] If $\sigma$ decomposes as a sum of two strongly orthogonal roots $\beta, \gamma \in \Phi^+$, then 
$$
	{\beta^\vee}_{|\Xi_{G/K}} = 	{\gamma^\vee}_{|\Xi_{G/K}} .
$$
Moreover, any such decomposition of $\sigma$ yields the same result.
\end{itemize}
\end{lemma}

\begin{proof}
i) It is well known that $\langle \lambda, \alpha \rangle \in 2 \ZZ$ whenever $2 \alpha \in \Delta^\rmsc_{G/K}$ (see \cite[Section 1.4]{Lu1}). It is indeed being part of Luna's axioms which classify spherical varieties \cite[Section 1.2]{BL}. Thus by the definition of $\Delta^\dag_{G/K}$ this property holds for all $\lambda \in \Xi_{G/K}$.

ii) We point out that in \cite[Lemmas 6.2 and 6.4]{KSc} a particular decomposition of $\sigma$ as a sum of two strongly orthogonal positive roots is considered, and the statement is proved for that decomposition.

Since $\Lambda^+_{G/K} = \Lambda^+_{G/\bar K}$ and $\Delta^{\rmsc}_{G/K} = \Delta^{\rmsc}_{G/\bar K}$, we may assume that $K =\bar K$. 

Scrolling through the list of Table~\ref{tab:sc_sph_roots}, we see that the $\rmsc$-spherical roots that are neither positive roots nor the double of a simple root are those of the following table (we also report an extra datum, which is defined here below).

\begin{table}[H]
\begin{tabular}{clc}
$\sigma$ & type of $\sigma$ & $\mathrm{supp}(\sigma)\setminus\Delta^{pp}(\sigma)$ \\
\hline
$\alpha+\alpha'$ & $\sfA_1\times\sfA_1$ & $\{\alpha,\alpha'\}$ \\
$\alpha_1+2\alpha_2+\alpha_3$ & $\sfA_3$ & $\{\alpha_2\}$ \\
$2(\alpha_1+\ldots+\alpha_m)$ & $\sfB_m$, $m\geq2$ & $\{\alpha_1\}$ \\
$\alpha_1+2\alpha_2+3\alpha_3$ & $\sfB_3$ & $\{\alpha_3\}$ \\
$2(\alpha_1+\ldots+\alpha_{m-2})+\alpha_{m-1}+\alpha_m$ & $\sfD_m$, $m\geq4$ & $\{\alpha_1\}$ \\
$4\alpha_1+2\alpha_2$ & $\sfG_2$ & $\{\alpha_1\}$\\
\end{tabular}
\end{table}

The first case, with $\sigma$ of type $\sfA_1\times\sfA_1$, is somewhat special and the statement is well known (see \cite[Proposition 3.2]{Lu1}). It is indeed part of Luna's axioms to classify spherical varieties. Therefore we will assume that $\supp(\sigma)$ is connected.

Denote
$$\Delta_{G/K}^p=\{\alpha\in \Delta\ :\ \langle\lambda,\alpha^\vee \rangle=0\ \forall\ \lambda\in\Omega^+_{G/K}\}$$
and let $\Delta^{pp}(\sigma) \subset \supp(\sigma)$ be the set of simple roots orthogonal to $\sigma$. Then we have $\Delta^{pp}(\sigma)\subset\Delta^p_{G/K}$ whenever $\sigma\not \in \Phi^+$ (see \cite[Section~1.1]{BL}).

Notice that in all cases of the previous table with $\supp(\sigma)$ connected there exists a unique simple root $\alpha \in \supp(\sigma)$ not orthogonal to $\sigma$. Therefore it is enough to check that, for any spherical root $\sigma \in \Delta^\rmsc_{G/K} \setminus (\Phi^+ \cup 2\Delta)$ and for any decomposition $\sigma=\beta+\gamma$ as a sum of two strongly orthogonal positive roots, it holds  
$$
	\langle\omega_\alpha,\beta^\vee\rangle=\langle\omega_\alpha,\gamma^\vee\rangle = 1
$$
which can easily be checked case-by-case. 
\end{proof}

For $\sigma\in\Delta_{G/K}$, using Lemma \ref{lemma:coradici sferiche} we define the coroot $\sigma^\vee\in \Xi_{G/K}^\vee$ as follows:
\begin{itemize}
	\item[i)] if $\sigma= 2\alpha$ with $\alpha\in\Delta$, then 
$$
	\sigma^\vee = \tfrac{1}{2}{\alpha^\vee}_{|\Xi_{G/K}}
$$

	\item[ii)] if $\sigma=\beta+\gamma$ with $\beta,\gamma\in\Phi^+$ strongly orthogonal, then 
$$
	\sigma^\vee = {\beta^\vee}_{|\Xi_{G/K}}
$$

	\item[iii)] if $\sigma=\beta$ with $\beta\in\Phi^+$, then
$$
	\sigma^\vee = {\beta^\vee}_{|\Xi_{G/K}}
$$
\end{itemize}

Set $\Delta_{G/K}^\vee = \{\sigma^\vee \; | \; \sigma \in \Delta_{G/K}\}$, then by construction we have the following.

\begin{proposition} \label{prop:root_datum}
The triple $\calR_{G/K} = (\Xi_{G/K}, \Delta_{G/K}, \Delta_{G/K}^\vee)$ is a based root datum. Moreover, the intersection $\Xi_{G/K} \cap \Lambda^+$ is contained in the monoid of dominant weights of $\calR_{G/K}$
\[\Xi_{G/K}^+ :=\{\lambda\in\Xi_{G/K}\ :\ \langle\lambda,\sigma^\vee\rangle\geq0\ \forall\ \sigma\in \Delta_{G/K}\}. \] 
\end{proposition}

Differently from the symmetric case, notice that the based root datum $\calR_{G/K}$ is not necessarily semisimple: by Remark \ref{oss:norm-reticoli}, this happens if and only if $K$ has finite index in its normalizer. When $\calR_{G/K}$ is semisimple every spherical weight is uniquely determined by its values against the spherical roots, otherwise the extra information is obtained by restricting to the normalizer. 

\begin{lemma}
Every $\lambda \in \Lambda_{G/K}$ is uniquely determined by its values against the spherical roots, together with its restriction to $B \cap \rmN_G(K)^\circ$.
\end{lemma}

\begin{proof}
Denote $K'= \rmN_G(K)^\circ$, then the restriction of characters to $B \cap K'$ induces an exact sequence
$$
	0 \longrightarrow \Lambda_{G/K'} \longrightarrow \Lambda_{G/K} \longrightarrow \calX(B \cap K')^{B \cap K} \longrightarrow 0
$$
(see e.g.\ \cite[Lemma 2.4]{Ga}). On the other hand  $\Lambda_{G/\bar K}\subset \Lambda_{G/K'}$ is a sublattice of finite index, and $\Lambda_{G/\bar K} = \ZZ \Delta^\rmsc_{G/K}$, thus every element in $\Lambda_{G/ K'}$ is uniquely determined by its values against the spherical roots.
\end{proof}

%%%%%%%%%%%%%%%%%%%%%%%%%%%%%%%%%%%%%%%
%%%%%%%%%%%%%%%%%%%%%%%%%%%%%%%%%%%%%%%
\section{Affine spherical varieties with root system of type A}	\label{sec:casi_tipo_A}
%%%%%%%%%%%%%%%%%%%%%%%%%%%%%%%%%%%%%%%
%%%%%%%%%%%%%%%%%%%%%%%%%%%%%%%%%%%%%%%

We now consider generalizations of Conjecture \ref{conj:stanley_cor} to other affine spherical homogeneous varieties for a simple algebraic group.

Let $(G,K)$ be a reductive spherical pair and set $X = G/K$. Let $G_X$ be the connected reductive group defined by the root datum $\calR_X = (\Xi_X, \Delta_X, \Delta_X^\vee)$, and correspondingly let $T_X \subset G_X$ and $B_X \subset G_X$ be a maximal torus and a Borel subgroup containing $T_X$ inducing $\calR_X$. For a spherical weight $\pi \in \Xi_X^+$ (in particular, for $\pi \in \Lambda_X^+$), let $V_X(\pi)$ be the irreducible $G_X$-module with highest weight $\pi$. In analogy with the symmetric case, we denote by $\leq_X$ the partial order on $\Xi_X$ defined by $\Delta^\rmn_X$, namely
$$
	\pi_1 \leq_X \pi_2 \stackrel{\mathrm{def}}{\Longleftrightarrow} \pi_2 - \pi_1 \in \NN \Delta^\rmn_X.
$$

\begin{conjecture} 	\label{conj:affine_spherical_A}
Let $(G,K)$ be a reductive spherical pair with $G$ simple and with $\Phi_{G/K}$ of type $\sfA$, and set $X = G/K$. Let $\lambda, \mu, \nu \in \Lambda_X^+$, then
$$
E_X(\nu) \subset E_X(\lambda) \cdot E_X(\mu) \Longleftrightarrow
\left\{ \begin{array}{c}
V_X(\nu) \subset V_X(\lambda) \otimes V_X(\mu)  \phantom{\Big|} \\
\nu \leq_X \lambda + \mu
\end{array} \right.
$$
\end{conjecture}

With one exception (case Sph.A15 of Table \ref{tab:sph_pairs_A}), we will show that the previous conjecture holds true whenever $\Phi_X$ is a direct sum of rank one root systems - which happens in most of the cases - as a consequence of Stanley's Pieri rule for Jack symmetric functions \cite[Theorem 6.1]{St}. More generally, we will show that in all cases but case Sph.A15 the previous conjecture follows from Conjecture \ref{conj:stanley_cor} (and in particular from Stanley's Conjecture  \ref{conj:stanley}). Finally, in the remaining case Sph.A15 the conjecture is supported by computational experiments.

For the remainder of the section, $G$ will be a simple and simply connected algebraic group, and $K \subset G$ a connected reductive spherical subgroup such that the spherical root system $\Phi_{G/K}$ is of type $\sfA$. A list of such spherical pairs is given in Table \ref{tab:sph_pairs_A} (where we omit the cases of Table \ref{tab:sym_pairs_A}). For a full list of the reductive spherical pairs $(G,K)$ with $G$ simple and $K$ connected we refer to \cite{Kra} and \cite{KR}.

\begin{table}
\caption{Reductive spherical pairs $(G,K)$ with root system of type $\sfA$}	\label{tab:sph_pairs_A}
 \begin{tabular}{|c|c|c|c|c|} 
\hline 
& $\phantom{\Big|}$ $G$ & $K$ & $\Phi_{G/K}$ \\
    \hline    \hline
Sph.A6 & $\phantom{\Big|}$ 	$\SL(n)$, $n\geq3$ & $\SL(n-1)$	 & $\sfA_1$ \\
\hline
Sph.A7 & $\phantom{\Big|}$ 	$\SL(n)$, $n\geq3$ & $\GL(n-1)$	 & $\sfA_1$ \\
\hline
Sph.A8 & $\phantom{\Big|}$ 	$\SL(2n+1)$, $n\geq2$ & $\Sp(2n)$ & $\sfA_n \times \sfA_{n-1}$ \\
\hline
Sph.A9 & $\phantom{\Big|}$ 	$\SL(2n+1)$, $n\geq2$ & $\GL(1) \times \Sp(2n)$ & $\sfA_n \times \sfA_{n-1}$ \\
\hline \hline
Sph.A10 & $\phantom{\Big|}$ 	$\Sp(2n)$, $n\geq2$ & $\GL(1) \times \Sp(2n-2)$	 & $\sfA_1 \times \sfA_1$ \\
    \hline
Sph.A11 & $\phantom{\Big|}$ 	$\Sp(2n)$, $n\geq3$ & $\Sp(2) \times \Sp(2n-2)$	 & $\sfA_1$ \\
\hline \hline
Sph.A12 & $\phantom{\Big|}$ $\Spin(7)$ & $\sfG_2$ & $\sfA_1$ \\
	\hline
Sph.A13 & $\phantom{\Big|}$ $\Spin(9)$ & $\Spin(7)$ & $\sfA_1 \times \sfA_1$ \\
  \hline \hline
Sph.A14 & $\phantom{\Big|}$ 	$\Spin(8)$ & $\sfG_2$ & $\sfA_1 \times \sfA_1 \times \sfA_1$ \\
   \hline \hline
Sph.A15 & $\phantom{\Big|}$ $\sfF_4$ & $\Spin(9)$	 & $\sfA_1$ \\
    \hline \hline
Sph.A16 & $\phantom{\Big|}$ $\sfG_2$ & $\SL(3)$	 & $\sfA_1$ \\
    \hline
\end{tabular}
\end{table}

\begin{remark}
As already pointed out in Remark \ref{oss:root_data} for the symmetric case, also in this case it is possible to formulate Conjecture \ref{conj:affine_spherical_A} using the based root datum $\mathcal R^\rmn_X = (\Lambda_X, \Delta^\rmn_X, {\Delta^\rmn_X}^{\!\!\vee})$ rather than $\mathcal R_X$, and defining the reductive group $G_X$ accordingly. Using $\mathcal R^\rmn_X$ the formulation becomes also simpler, because the inequality $\nu\leq_X \lambda+\mu$ becomes redundant. However the formulation that we adopted seems to be preferable, as it fits inside the more general framework that we will consider in Section \ref{sec:Levi}. Notice also that in all cases but Sph.A10 it holds $\Delta_X = \Delta_X^\rmn$, so that in these cases we have indeed $\mathcal R_X = \mathcal R_X^\rmn$.
\end{remark}

\begin{remark}
Notice that the cases Sph.A7, Sph.A11 and Sph.A15 are actually symmetric, with non-reduced restricted root system of type $\sfBC_1$.

Notice also that the unique cases where $\Phi_{G/K}$ is not a direct sum of subsystems of type $\sfA_1$ are cases Sph.A8 and Sph.A9. In the first case, $G/K$ is the \textit{principal model variety} of $\SL(n)$: a homogeneous variety $G/H$ is said to be a model variety for $G$ if it is spherical and $\Lambda_{G/K}^+ = \calX(T)^+$ (see \cite{GZ1}, \cite{GZ2}). In the latter case, $G/K$ is the homogeneous space which correspond to the model variety of $\SL(n)$ in Luna's classification \cite{Lu2}.

All other cases, where $\Phi_{G/K}$ is a direct sum of subsystems of type $\sfA_1$, are related to some transitive action of a compact Lie group on a projective space, on a sphere, or on a product of spheres. Let indeed $G_\RR \subset G$ and $K_\RR \subset K$ be the associated compact real forms, then we have the following descriptions for $G_\RR/K_\RR$:
\begin{itemize}
		\item[i)] it is isomorphic to $\PP^n(\CC)$ in case Sph.A7, to $\PP^{2n}(\CC)$ in case Sph.A10, to $\PP^n(\HH)$ in case Sph.A11, to $\PP^2(\OO)$ in case Sph.A15 (where $\HH$ and $\OO$ denote respectively the algebras of the quaternions and of the octonions);
		\item[ii)] it is isomorphic to $\mathbb S^{2n-1}$ in case Sph.A6, to $\mathbb S^7$ in case Sph.A12, to $\mathbb S^{15}$ in case Sph.A13, to $\mathbb S^6$ in case Sph.A16, to $\mathbb S^7 \times \mathbb S^7$ in case Sph.A14.
\end{itemize}

%Transitive actions of compact Lie groups on a sphere are classified. 
By a theorem of Cartan, if $G_\RR/H_\RR$ is an irreducible compact Riemannian symmetric space of rank one, then $H_\RR$ acts transitively on the unit sphere in the isotropy representation of $H_\RR$. This gives indeed rise to the transitive actions on spheres which correspond to cases Sph.A6 and Sph.A13 (which respectively come from the families of symmetric spaces corresponding to cases Sph.A7 and Sph.A15). More generally, the transitive actions of a Lie group on a sphere were classified by Montgomery-Samelson and Borel, those which are not related to a symmetric space of rank one correspond to cases Sph.A12 and Sph.A13 (see e.g. \cite{STV}).
		
There is indeed another transitive action on a sphere which is associated to a rank one symmetric space, which is not included in Table \ref{tab:sph_pairs_A}: it is the transitive action of $\Sp(2,\RR) \times \Sp(2n,\RR)$ on $\mathbb S^{4n-1}$ (which comes from the family of symmetric spaces corresponding to case Sph.A11). This action corresponds to the reductive spherical pair
$$
	(G,K) = \big(\Sp(2) \times \Sp(2n), \ \diag(\Sp(2)) \times \Sp(2n-2) \big),
$$
in which case $\Phi_{G/K}$ is of type $\sfA_1 \times \sfA_1$. Even though $G$ is not simple, this case also fits in the general framework which allows us to deal with the cases of Table \ref{tab:sph_pairs_A}. We will analyze this case separately in Remark \ref{oss:sfera_transitiva_nonsemplice}.

%\begin{table}
%\caption{Reductive spherical pairs $(G,K)$ with root system of type $\sfA$ and related spheres}
% \begin{tabular}{|c|c|c|c|c|} 
%\hline 
%& $\phantom{\Big|}$ $G_\RR$ & $K_\RR$ & $G_\RR/K_\RR$ \\
%    \hline    \hline
%Sph.A6 & $\phantom{\Big|}$ 	$\SU(n)$, $n\geq3$ & $\SU(n-1)$	 & $S^{2n-1}$ \\
%\hline
%Sph.A7 & $\phantom{\Big|}$ 	$\SU(n)$, $n\geq3$ & $\U(n-1)$	 & $\PP^n(\CC)$ \\
%\hline
%Sph.A10 & $\phantom{\Big|}$ 	$\Sp_\RR(2n)$, $n\geq2$ & $\U(1) \times \Sp_\RR(2n-2)$	 & $\PP^{2n}(\CC)$ \\
%    \hline
%Sph.A11 & $\phantom{\Big|}$ 	$\Sp_\RR(2n)$, $n\geq3$ & $\Sp_\RR(2) \times \Sp_\RR(2n-2)$	 & $\PP^n(\HH)$ \\
%\hline \hline
%Sph.A12 & $\phantom{\Big|}$ $\Spin_\RR(7)$ & $\sfG_2$ & $S^7$ \\
%	\hline
%Sph.A13 & $\phantom{\Big|}$ $\Spin_\RR(9)$ & $\Spin(7)$ & $S^{15}$ \\
%  \hline \hline
%Sph.A14 & $\phantom{\Big|}$ 	$\SO_\RR(8)$ & $\sfG_2$ & $S^7 \times S^7$ \\
%   \hline \hline
%Sph.A15 & $\phantom{\Big|}$ $\sfF_4$ & $\Spin_\RR(9)$	 & $\PP^2(\OO)$ \\
%    \hline \hline
%Sph.A16 & $\phantom{\Big|}$ $\sfG_2$ & $\SU(2)$	 & $S^6$ \\
%    \hline
%\end{tabular}
%\label{tab:sph_pairs_A}
%\end{table}

\end{remark}

%%%%%%%%%

To deal with the cases of Table \ref{tab:sph_pairs_A}, a first useful remark is that cases Sph.A7, Sph.A9 and Sph.A11 respectively follow (almost immediately) from the companion cases Sph.A6, Sph.A8 and Sph.A10. This will be explained in Section \ref{ssec:other cases}.

All the remaining cases of Table \ref{tab:sph_pairs_A} - with the exception of case Sph.A15 - will follow from some cases of Table \ref{tab:sym_pairs_A}, thanks to a related factorization of a simple algebraic group as a product of two reductive subgroups. Such factorizations have been classified by Onishchik \cite{On} in the context of compact Lie groups (see also \cite[Section 4.5-4.6]{GO}) and by Liebeck, Saxl and Seitz algebraically in arbitrary characteristic \cite{LSS}. More precisely, in all the considered cases the multiplication will reduce to that of some symmetric varieties, which correspond to a decomposition of $\Phi_{G/K}$ into two orthogonal components. 

%The remaining cases be treated in two different ways, in terms of factorizations of reductive group. Indeed there are strong connections between the cases listed in Table \ref{tab:sph_pairs_A} and the classification of the factorizations of a simple group as a product of two reductive subgroups \cite{On}.

In particular, we will find a reductive overgroup $\widehat G \supset G$ and a symmetric subgroup $\widehat K \subset \widehat G$ such that $\widehat G = G \cdot \widehat K$ and $K = G \cap \widehat K$, yielding a $G$-equivariant isomorphism
$$
	G/K \simeq \widehat G/\widehat K
$$
and inducing a natural projection of weight monoids
$$
	\hat \varphi: \Lambda^+_{G/K} \longrightarrow \Lambda^+_{\widehat G/\widehat K} .
$$

As a consequence of a theorem of Luna \cite{Lu0}, notice that in the previous setting $G/K$ and $\widehat G/\widehat K$ are isomorphic as $G$-varieties if and only if $G/K$ embeds $G$-equivariantly into $\widehat G/\widehat K$ as a dense open subset.

We will also need to bring into the picture a second symmetric pair $(H,K)$ (possibly trivial), where $H$ is a suitable connected reductive subgroup of $G$ containing $K$ and the connected center of $\rmN_G(K)$, yielding a $H$-equivariant embedding
$$
	H/K \hookrightarrow G/K
$$
and inducing a second projection of weight monoids
$$
	\bar \varphi:\Lambda^+_{G/K} \longrightarrow \Lambda^+_{H/K}.
$$

%In cases Sph.A12 and Sph.A16 the map $\varphi_{G^0}$ will actually identify the pairs $(\Lambda_{G^0/K^0}, \Phi_{G^0/K^0})$ and  $(\Lambda_{G/K}, \Phi_{G/K})$, and the isomorphism $G/K \simeq G^0/K^0$ will be sufficient to identify the spherical functions of $(G,K)$, and to reduce the their multiplication to that of the spherical functions of the symmetric pair $(G^0, K^0)$.
%
%In the other cases we will 

The two maps $\hat \varphi$ and $\bar \varphi$ will allow us to identify the spherical functions of $G/K$, and describe their products, in terms of those of $\widehat G/\widehat K$ and $H/K$. Before treating the various cases in detail, let us explain here the general argument.

Suppose that we have:
\begin{itemize}
\item[i)] a connected reductive overgroup $\widehat G \supset G$ and a spherical subgroup $\widehat K \subset \widehat G$ such that $G/K \simeq \widehat G/\widehat K$,
\item[ii)] a connected reductive subgroup $H\subset G$ containing $K$ (and the connected center of $\rmN_G(K)$).
\end{itemize}

We also assume (it will always be true in our cases) that there is a compatible choice of Borel subgroups, that is, there exists a maximal torus and Borel subgroup $\widehat T \subset \widehat B$ of $\widehat G$ with $\widehat B \widehat K$ open in $\widehat G$ such that 
\begin{itemize}
	\item[i)] $T = \widehat T \cap G$ is a maximal torus of $G$, and $B = \widehat B \cap G$ is a Borel subgroup of $G$ with $B K$ is open in $G$,
	\item[ii)] $T_H = T \cap H$ is a maximal torus of $H$, and $B_H = B \cap H$ is a Borel subgroup of $H$ with $B_H K$ is open in $H$.
\end{itemize}

On the one hand the restriction of the $\widehat G$-action to $G$ induces an isomorphism of $G$-modules
$$
\bigoplus_{\mu \in \Lambda^+_{\widehat G/\widehat K}} E_{\widehat G/\widehat K}(\mu) \; = \; 
\CC[\widehat G/\widehat K] \stackrel{\sim}{\longrightarrow} 
\CC[G/K] \quad = 
\bigoplus_{\lambda \in \Lambda^+_{G/K}} E_{G/K}(\lambda) 
$$
Thus for all $\lambda \in \Lambda^+_{G/K}$ there exists a unique $\hat \lambda \in \Lambda^+_{\widehat G/\widehat K}$ such that 
$$
	\Hom_G\big(V_{\widehat G}\big(\hat \lambda\big),V_G(\lambda)\big) \neq 0,
$$
and we define $\hat\varphi(\lambda) = \hat \lambda$.

On the other hand, by definition, for all $\lambda \in \Lambda_{G/K}^+$ the invariant space $V_G(\lambda)^K$ is one dimensional. Looking at the decomposition of $V_G(\lambda)$ into simple $H$-modules, it follows that there exists a unique $\bar \lambda \in \Lambda_{H/K}^+$ such that
$$
	\Hom_{H}\big(V_G(\lambda), V_{H}\big(\bar \lambda\big)\big) \neq 0
$$
(and the dimension of the latter is necessarily one). Thus we define $\bar \varphi(\lambda) = \bar \lambda$.

Consider now the restriction of functions from $G/K$ to $H/K$
$$
\bigoplus_{\lambda \in \Lambda^+_{G/K}} E_{G/K}(\lambda) \; = \;
	\CC[G/K] \stackrel{\rho}{\longrightarrow}
\CC[H/K] \quad =
\bigoplus_{\mu \in \Lambda^+_{H/K}} E_{H/K}(\mu) .
$$
In the following lemma we show that $\rho(E_{G/K}(\lambda))$ is non-zero, hence it coincides with $E_{H/K}(\bar \lambda)$.

\begin{lemma}
Let $V \subset \CC[G/K]$ be an irreducible $G$-submodule. Then $\rho(V) \subset \CC[H/K]$ is an irreducible $H$-submodule.
\end{lemma}

\begin{proof}
By the discussion above we only need to show that $\rho(V) \neq 0$.

Suppose that $V \simeq V_G(\lambda)$. Fix $K$-invariant vectors $v_\lambda \in V_G(\lambda)$ and $\psi_\lambda \in V_G(\lambda)^*$, and define
$$
	f_{G,\lambda}(gK) = \langle \psi_\lambda, g.v_\lambda \rangle .
$$
Then $f_{G,\lambda} \in V^K$ is a non-zero $K$-invariant function.

Similarly, if $v_{\bar \lambda} \in V_{H}(\bar \lambda)$ and $\psi_{\bar \lambda} \in V_{H}(\bar \lambda)^*$ are $K$-invariant vectors, 
setting
$$
	f_{H,\bar \lambda}(hK) = \langle \psi_{\bar \lambda},h.v_{\bar \lambda} \rangle
$$
we get a non-zero $K$-invariant function on $H/K$.

On the other hand, there is a unique $H$-equivariant embedding of $V_{H}(\bar \lambda)$ inside $V_G(\lambda)$, and a unique $H$-equivariant embedding of $V_{H}(\bar \lambda)^*$ inside $V_G(\lambda)^*$. With respect to such embeddings, we have equalities
$$
V_G(\lambda)^K = V_{H}(\bar \lambda)^K
\qquad \qquad
(V_G(\lambda)^*)^K = (V_{H}(\bar \lambda)^*)^K .
$$
Since these invariant subspaces are one dimensional, it follows that the $K$-invariant vectors $v_\lambda$ and $v_{\bar \lambda}$ are proportional, and similarly for the $K$-invariant linear functions $\psi_\lambda$ and $\psi_{\bar \lambda}$. Thus the restriction $\rho(f_{G,\lambda})$ is a non-zero scalar multiple of $f_{H,\bar \lambda}$.
\end{proof}

We now show that the two maps $\hat \varphi$ and $\bar \varphi$, defined above, are actually morphisms of monoids.

\begin{proposition}
With notations and hypotheses as above, both the maps 
	$$\hat \varphi : \Lambda^+_{G/K} \longrightarrow \Lambda^+_{\widehat G/\widehat K}	 
\quad\mbox{and}\quad 
	\bar \varphi : \Lambda^+_{G/K} \longrightarrow \Lambda^+_{H/K}$$ 	
are additive maps. 
\end{proposition}

\begin{proof}
i) Let $\Gamma(\widehat G,G)$ be the \textit{branching monoid} associated to the pair $(\widehat G,G)$, namely
$$
	\Gamma(\widehat G,G) = \{(\pi,\lambda) \in \calX(\widehat T)^+ \times \calX(T)^+ \; | \; \Hom_{G}(V_G(\lambda), V_{\widehat G}(\pi)) \neq 0 \}
$$
This is indeed a finitely generated monoid, see \cite{AP}.

Consider the projections
$$
	p : \Gamma(\widehat G,G) \rightarrow \calX(\widehat T)^+,
	\qquad \qquad
	q : \Gamma(\widehat G,G) \rightarrow \calX(T)^+.
$$
Since $K$ is spherical in $G$, notice that $p^{-1}(\Lambda_{\widehat G/\widehat K}^+)$ is a submonoid of $\Gamma(\widehat G,G)$, and that $q$ induces an isomorphism $p^{-1}(\Lambda_{\widehat G/\widehat K}^+) \rightarrow \Lambda_{G/K}^+$.

If indeed $(\pi,\lambda) \in \Gamma(\widehat G,G)$ and $\pi\in \Lambda_{\widehat G/\widehat K}^+$, then by the isomorphism $\CC[G/K] \simeq \CC[\widehat G/\widehat K]$ it follows $\lambda \in \Lambda_{G/K}^+$. Conversely, if $\lambda \in \Lambda_{G/K}^+$, since $\CC[G/K] \simeq \CC[\widehat G/\widehat K]$ the multiplicity-free property implies that there is a unique $\pi \in \Lambda_{\widehat G/\widehat K}^+$ with $(\pi,\lambda) \in \Gamma(\widehat G,G)$, and by definition $\pi = \hat \varphi(\lambda)$. 

It is now clear that, under the isomorphism $\Lambda_{G/K}^+ \simeq p^{-1}(\Lambda_{\widehat G/\widehat K}^+)$, the map $\hat \varphi$ coincides with the restriction of the projection $p$. In particular, we see that $\hat \varphi$ is an additive map.

ii) The proof that $\bar \varphi$ is additive is similar to the previous one, by considering the branching monoid associated to the pair $(G,H)$
$$
	\Gamma(G,H) =  \{(\lambda, \pi) \in \calX(T)^+ \times \calX(T_H)^+ \; | \; \Hom_{H}(V_H(\pi), V_{G}(\lambda)) \neq 0 \}
$$ 
 together with the projections 
$$
	p : \Gamma(G,H) \rightarrow \calX(T)^+,
	\qquad \qquad
	q : \Gamma(G,H) \rightarrow \calX(T_H)^+.
$$
Then $p$ induces an isomorphism of monoids $q^{-1}(\Lambda_{H/K}^+) \rightarrow \Lambda_{G/K}^+$, and $\bar \varphi$ corresponds to the restriction of $q$.
%
%ii) DIMOSTRAZIONE COMPLETA 
%Let $\Gamma(G,G_0)$ be the branching monoid associated to the pair $(G,G_0)$, namely
%$$
%	\Gamma(G,G_0) = \{(\lambda, \pi) \in \calX(T)^+ \times \calX(T_0)^+ \; | \; \Hom_{G_0}(V_{G_0}(\pi),V_G(\lambda)) \neq 0 \}
%$$
%Since $K$ is spherical in $G$, notice that
%$$
%\Gamma(G,G_0) \cap (\calX(T)^+ \times \Lambda_{G_0/K}^+) \subset \Lambda_{G/K}^+ \times \Lambda_{G_0/K}^+
%$$
%is isomorphic to $\Lambda_{G/K}^+$. If indeed $(\lambda, \pi) \in \Gamma(G,G_0)$ and $\pi$ is a $K$-spherical weight, then $\lambda$ is a $K$-spherical weight as well. Conversely, if $\lambda \in \Lambda_{G/K}^+$ then there exists a unique $\pi \in \Lambda_{G_0/K}^+$ such that $(\lambda,\pi) \in \Gamma(G,G_0)$, and by definition $\pi = \varphi_{G_0}(\lambda)$. On the other hand if $\lambda, \mu \in \Lambda_{G/K}^+$ then $(\lambda+\mu,\ \varphi_{G_0}(\lambda) + \varphi_{G_0}(\mu)) \in \Gamma(G,G_0)$ the latter is a monoid. Thus we get the equality
%\[
%	\varphi_{G_0}(\lambda+\mu) = \varphi_{G_0}(\lambda) + \varphi_{G_0}(\mu). \qedhere
%\]
%
\end{proof}

Therefore, we can extend $\hat \varphi$ and $\bar \varphi$ to morphisms of the associated weight lattices. When the lattice $\Xi_{G/K}$ is bigger than $\Lambda_{G/K}$, we will see in the cases under consideration that the extensions of $\hat \varphi$ and $\bar \varphi$ to $\Xi_{G/K}$ respectively take values in $\Xi_{\widehat G/\widehat K}$ and $\Xi_{H/K}$.  The description of the maps $\hat \varphi$ and $\bar \varphi$ will immediately imply the following.

\begin{proposition}\label{prop:isogeny}
Let $(G,K)$ be a reductive spherical pair with $\Phi_{G/K}$ of type $\sfA$, and suppose that we are in one of the following cases of Table \ref{tab:sph_pairs_A}:
$$
\mathrm{Sph.A6, \; Sph.A8, \; Sph.A10, \; Sph.A12, \; Sph.A13, \; Sph.A14, \; Sph.A16}.$$
Then there exist a connected reductive subgroup $H \subset G$ containing $K$ (possibily equal to $K$), a connected reductive overgroup $\widehat G \supset G$ and a symmetric subgroup $\widehat K \subset \widehat G$ with $\widehat K \cap G = K$ such that the map
$$(\hat \varphi,\bar \varphi) \; \colon \; \Xi_{G/K} \longrightarrow \Xi_{\widehat G/\widehat K} \oplus  \Xi_{H/K}$$
is an isogeny of based root data
%
%of the based root datum $\Phi_{G/K}$ associated to the spherical $G$-variety $G/K$ into the based root datum associated to the spherical $G^0 \times G_0$-variety $G^0/K^0 \times G_0/K$.
%
\[
\calR_{G/K} \longrightarrow \calR_{\widehat G/\widehat K} \oplus  \calR_{H/K} .
\]
\end{proposition}

%%%%%

To be clear about notation and terminology, $\calR_{\widehat G/\widehat K} \oplus  \calR_{H/K}$ is the based root datum defined by
$$\big(\Xi_{\widehat G/\widehat K} \oplus  \Xi_{H/K},\ \Delta_{\widehat G/\widehat T}\sqcup \Delta_{H/K},\ \Delta_{\widehat G/\widehat T}^\vee\sqcup \Delta_{H/K}^\vee\big)$$ 
and $\varphi=(\hat \varphi,\bar \varphi)$ is an \textit{isogeny of based root data} in the following (usual) sense:
\begin{itemize}
	\item[i)] $\varphi : \Xi_{G/K} \rightarrow \Xi_{\widehat G/\widehat K} \oplus  \Xi_{H/K}$ is injective with finite cokernel, 
	\item[ii)] $\varphi$ restricts to a bijection between $\Delta_{G/K}$ and $\Delta_{\widehat G/\widehat T}\sqcup \Delta_{H/K}$, 
	\item[iii)] the dual map satisfies $\varphi^\vee(\varphi(\sigma)^\vee)=\sigma^\vee$, for all $\sigma \in \Delta_{G/K}$.
\end{itemize}

In particular, we stress the following property.

\begin{corollary}	\label{cor:injectivity}
Every weight in $\Lambda_{G/K}$ is uniquely determined by its images in $\Lambda_{\widehat G/\widehat K}$ and in $\Lambda_{H/K}$.
\end{corollary}

In the next two subsections we will deal with the various cases of Proposition \ref{prop:isogeny}. 

For every spherical pair $(G,K)$ we provide the groups $\widehat G$, $\widehat K$ and $H$ with  the required properties. The inclusions among the given groups are well known and all the mentioned properties can be read off from the classification of the factorizations of simple algebraic groups \cite{On}, \cite{LSS} (see also \cite[Section 4.5-4.6]{GO}).

For each affine spherical variety $G/K$, $\widehat G/\widehat K$ and $H/K$ we describe the sets of spherical roots $\Delta_{G/K}$, $\Delta_{\widehat G/\widehat K}$ and $\Delta_{H/K}$ (which we call \emph{normalized spherical roots}), and minimal set of generators of the monoids $\Lambda^+_{G/K}$, $\Lambda^+_{\widehat G/\widehat K}$, $\Lambda^+_{H/K}$ (which we call \emph{fundamental spherical weights}). These data are also well known and can be read off for example from \cite{BPe}. 

Finally, in each case, we describe the maps $\hat \varphi$ and $\bar \varphi$. Being additive, it is enough to describe the images of the fundamental spherical weights of $G/K$.

%%%%%%%%%%%%%%%%%%%%%%%%%%%%%%%%%%%%%%%
%%%%%%%%%%%%%%%%%%%%%%%%%%%%%%%%%%%%%%%
\subsection{Indecomposable cases}\label{ssec:indecomposable cases}
%%%%%%%%%%%%%%%%%%%%%%%%%%%%%%%%%%%%%%%
%%%%%%%%%%%%%%%%%%%%%%%%%%%%%%%%%%%%%%%

We consider here the cases Sph.A12 and Sph.A16 where we take the subgroup $H \subset G$ as $K$ itself (thus $\Xi_{H/K} = 0$ and $\Delta_{H/K} = \emptyset$). The overgroup $\widehat G \supset G$ and the symmetric subgroup $\widehat K \subset \widehat G$ with the prescribed properties are given below.

From its explicit description, it will be clear that the map
$$
	\hat \varphi : \Lambda_{G/K} \longrightarrow \Lambda_{\widehat G/\widehat K}
$$
induces an isomorphism of root data
\[
\calR_{G/K} \stackrel{\sim}{\longrightarrow} \calR_{\widehat G/\widehat K}.
\]

%Notice that we have standard inclusions $G \subset G^0$ and $K^0 \subset G^0$ with $K = K^0 \cap G$, giving rise to an isomorphism 
%$$
%	G/K \stackrel{\sim}{\longrightarrow} G^0/K^0
%$$
%(see e.g. \cite[Section 4.5-4.6]{GO}). Along the way, we also provide an explicit description of the map 
%$$
%	\varphi_{G^0}: \Lambda^+_{G/K} \longrightarrow \Lambda^+_{G^0/K^0} \qquad \qquad 
%\lambda \longmapsto \lambda^0
%$$
%which will identify the pairs $(\Lambda_{G/K}, \Phi_{G/K})$ and $(\Lambda_{G^0/K^0}, \Phi_{G^0/K^0})$.
%\footnote{Forse sarebbe meglio parlare proprio di isomorfismo di root data.}\\

\begin{itemize}
\item[\textbf{(Sph.A12)}] $G/K = \Spin(7) / \sfG_2$ 
$$\widehat G / \widehat K = \SO(8) / \SO(7)$$
%		
%		$$\Delta_{G/K} = \{\alpha_1 + 2 \alpha_2 + 3 \alpha_3\}
%\qquad \qquad
%			\Delta_{G^0/K^0} = \{2\alpha_1 + 2 \alpha_2 + \alpha_3 + \alpha_4\}$$
%		$$\Lambda_{G/K} = \ZZ \omega_3
%\qquad \qquad
%			\Lambda_{G^0/K^0} = \ZZ \omega_1'$$
%$$
%\begin{array}{ccccc}
%\varphi_{G^0} & : & a\omega_3 & \longmapsto & a\omega_1'
%\end{array}
%$$
%
%\begin{table}[H]
%\begin{tabular}{|c|c|}
%    \hline
%$\Spin(7)/\sfG_2$ & $\SO(8)/\SO(7)$
% \\
%    \hline \hline
%%%%%%%%   Riga 1
%\multirow{3}{*}{\begin{picture}(4000,600)(-300,-200)
%\put(0,0){\usebox{\bthirdthree}}
%\put(3100,420){\tiny1\!/2}
%\end{picture}} &
%\multirow{3}{*}{\begin{picture}
%(3800,800)(-300,-200)
%\put(0,0){\usebox{\dynkindfour}}
%\put(-500,420){\tiny1\!/2}
%\put(0,0){\usebox{\gcircle}}
%\end{picture}} \\
%& \\ & \\
%    \hline
%\end{tabular}
%\end{table}
%
$$
\begin{array}{|c|c|c|}
\hline
\phantom{\Big|} & \mbox{norm. spherical roots} & \mbox{fund. sph. weights}\\
\hline \hline
\phantom{\Big|}G/K\phantom{\Big|} & \alpha_1 + 2 \alpha_2 + 3 \alpha_3 & \omega_3 \\
\hline
\phantom{\Big|}\widehat G / \widehat K\phantom{\Big|} & 2(\alpha'_1 + \alpha'_2) + \alpha'_3 + \alpha'_4 & \omega'_1 \\
\hline
\end{array}
$$

$$\begin{array}{|ccccc|}
\hline
\phantom{\Big|} \hat \varphi & : & a\omega_3 & \longmapsto & a\omega_1' \\
\hline
\end{array}$$

\bigskip
\item[\textbf{(Sph.A16)}] $G / K = \sfG_2 / \SL(3)$ 
$$\widehat G / \widehat K = \SO(7) / \SO(6)$$
%
%		$$\Delta_{G/K} = \{4\alpha_1 + 2 \alpha_2\}
%\qquad \qquad
%			\Delta_{G^0/K^0} = \{2\alpha_1 + 2 \alpha_2 + 2\alpha_3 \}$$
%		$$\Lambda_{G/K} = \ZZ \omega_1
%\qquad \qquad
%			\Lambda_{G^0/K^0} = \ZZ \omega_1'$$
%			$$
%\begin{array}{ccccc}
%\varphi_{G^0} & : & a\omega_1 & \longmapsto & a\omega_1'
%\end{array}
%$$
%
%\begin{table}[H]
%\begin{tabular}{|c|c|}
%    \hline
%$\sfG_2/\SL(3)$ & $\SO(7)/\SO(6)$
% \\
%    \hline \hline
%%%%%%%%   Riga 1
%\multirow{2}{*}{\begin{picture}
%(1800,800)(-300,-300)
%\put(0,0){\usebox{\gtwo}}
%\end{picture}} &
%\multirow{3}{*}{\begin{picture}
%(3600,800)(-300,-900)
%\put(0,0){\usebox{\dynkinbthree}}
%\put(0,0){\usebox{\gcircle}}
%\end{picture}} \\
% & \\
%    \hline
%\end{tabular}
%\end{table}
%
$$
\begin{array}{|c|c|c|}
\hline
\phantom{\Big|} & \mbox{norm. spherical roots} & \mbox{fund. sph. weights}\\
\hline \hline
\phantom{\Big|} G/K \phantom{\Big|} & 4\alpha_1 + 2 \alpha_2 & \omega_1 \\
\hline
\phantom{\Big|}\widehat G / \widehat K\phantom{\Big|} & 2(\alpha'_1 + \alpha'_2 + \alpha'_3) &  \omega'_1 \\
\hline
\end{array}
$$

$$
\begin{array}{|ccccc|}
\hline
\phantom{\Big|} \hat \varphi & : & a\omega_1 & \longmapsto & a\omega_1' \\
\hline
\end{array}
$$

\end{itemize}

\bigskip
\begin{proposition}\label{prop:indec}
Conjecture \ref{conj:affine_spherical_A} holds true if $(G,K)$ is one of the spherical pairs Sph.A12, Sph.A16 of Table \ref{tab:sph_pairs_A}.
\end{proposition}

\begin{proof}
By restricting the action from $\widehat G$ to $G$, every spherical module for the pair $(\widehat G, \widehat K)$ remains irreducible, and yields a spherical module for the pair $(G, K)$. This gives a canonical bijection between the two: the restriction of every irreducible $\widehat G$-module containing a nontrivial $\widehat K$-invariant subspace is irreducibile as a $G$-module as well, and it is obviously $K$-spherical.

In particular, the decomposition of the product of two spherical functions of $G/K$ immediately reduces to the decomposition of the product of the corresponding spherical functions of $\widehat G/\widehat K$. 

%This immediately implies that the map
%$$
%	\varphi_{G^0}: \Lambda^+_{G/K} \longrightarrow \Lambda^+_{G^0/K^0} 
%$$
%is injective: indeed $\Lambda_{G/K} = \ZZ \Delta^{\min}_{G/K} = \tfrac{1}{2} \ZZ \Delta_{G/K}$, thus the weight $\lambda$ is uniquely determined by the coefficient of $\lambda_\res$ against the unique spherical root.

By Stanley's Pieri rule for Jack symmetric functions, we know that Conjecture \ref{conj:stanley_cor} holds true for the symmetric pair $(\widehat G,\widehat K)$. On the other hand $\hat \varphi$ induces an isomorphism of root data $\calR_{G/K} \simeq \calR_{\widehat G/\widehat K}$. Thus Conjecture \ref{conj:affine_spherical_A} holds true for $(G,K)$ as well.
\end{proof}

%%%%%%%%%%%%%%%%%%%%%%%%%%%%%%%%%%%%%%%
%%%%%%%%%%%%%%%%%%%%%%%%%%%%%%%%%%%%%%%
\subsection{Decomposable cases}
\label{ssec:decomposable cases}
%%%%%%%%%%%%%%%%%%%%%%%%%%%%%%%%%%%%%%%
%%%%%%%%%%%%%%%%%%%%%%%%%%%%%%%%%%%%%%%

We now consider the cases Sph.A6, Sph.A8, Sph.A10, Sph.A13 and Sph.A14 of Table \ref{tab:sph_pairs_A}: we define the subgroup $H \subset G$, the overgroup $\widehat G \supset G$ and the symmetric subgroup $\widehat K \subset \widehat G$ with the required properties. Notice that $H$ will always contain the connected center of $\rmN_G(K)$. As already mentioned, for the given inclusions among groups and their related properties one can see \cite{On}, \cite{LSS} and also \cite[Section 4.5-4.6]{GO}.

We also provide an explicit description of the maps
$$
	\hat \varphi: \Lambda_{G/K} \longrightarrow \Lambda_{\widehat G/\widehat K} 
\quad\mbox{and}\quad
	\bar \varphi: \Lambda_{G/K} \longrightarrow \Lambda_{H/K}.
$$
Proposition \ref{prop:isogeny} will be an immediate consequence of these descriptions.\\

\begin{itemize}
	
%%%%%%%%%%%%%%%%%%%%%%%%%%%%%%%%%%%%%%%%%
%%%%%%%%%%%%%%%%%%%%%%%%%%%%%%%%%%%%%%%%%
%%%%%%%%%%%   			Levi.A 				%%%%%%%%%%%
%%%%%%%%%%%%%%%%%%%%%%%%%%%%%%%%%%%%%%%%%
%%%%%%%%%%%%%%%%%%%%%%%%%%%%%%%%%%%%%%%%%

\item[\textbf{(Sph.A6)}] $G / K = \SL(n) / \SL(n-1)$ 
$$\widehat G / \widehat K = \SO(2n) / \SO(2n-1)$$ 
$$H / K = \GL(n-1) / \SL(n-1)$$
%
%		\begin{table}[H]
%\begin{tabular}{|c|c|c|}
%    \hline
%$\GL(n-1)/\SL(n-1)$ & $\SL(n)/\SL(n-1)$ & $\SO(2n)/\SO(2n-1)$
% \\
%    \hline \hline
%%%%%%%%   Riga 1
%\multirow{3}{*}{\begin{picture}
%(6000,800)(-1400,200)
%\put(0,0){\usebox{\susp}}
%\end{picture}} 
%& \multirow{3}{*}{\begin{picture}
%(6000,800)(-1400,200)
%\put(0,0){\usebox{\shortam}}
%\end{picture}} &
%\multirow{3}{*}{
%\begin{picture}(6900,2400)(-300,-600)
%\put(0,0){\usebox{\shortdm}}
%\put(-500,420){\tiny1\!/2}
%\end{picture}}  \\
%&  & \\ & & \\ & & \\
%    \hline
%\end{tabular}
%\end{table}
%
$$
\begin{array}{|c|c|c|}
\hline
\phantom{\Big|} & \text{normalized spherical roots} & \text{fund. sph. weights}\\
\hline \hline
\phantom{\Big|} G/K \phantom{\Big|} & \alpha_1 + \ldots + \alpha_{n-1} & \omega_1, \ \omega_{n-1} \\
\hline
\phantom{\Big|}\widehat G / \widehat K\phantom{\Big|} & 2(\alpha'_1 + \ldots + \alpha'_{n-2}) + \alpha'_{n-1} + \alpha'_n & \omega'_1 \\
\hline
\phantom{\Big|}H/K \phantom{\Big|} & \text{ none } & \pm\omega''_{n-1}\\
\hline
\end{array}
$$

$$
\begin{array}{|ccccl|}
\hline
\phantom{\Big|} \hat \varphi & : & a \omega_1 + b\omega_{n-1} & \longmapsto & (a+b)\omega'_1  \\ 
\hline
\phantom{\Big|} \bar \varphi & : & a \omega_1 + b\omega_{n-1} & \longmapsto  & (b-a) \omega''_{n-1} \\
\hline
\end{array}
$$

%		$$\Delta_{G/K} = \{\alpha_1 + \ldots + \alpha_{n-1}\}
%		\qquad \qquad
%		\Lambda_{G/K} = \ZZ \omega_1 + \ZZ\omega_{n-1}$$
%		$$	\Delta_{G^0/K^0} = \{2\alpha'_1 + \ldots + 2\alpha'_{n-2} + \alpha'_{n-1} + \alpha'_n\}
%		\qquad \qquad
%		\Lambda_{G^0/K^0} = \ZZ \omega'_1$$
%$$		\Delta_{G_0/K} = \emptyset
%\qquad \qquad
%	\Lambda_{G_0/K} = \ZZ \chi
%$$

%%%%%%%%%%%%%%%%%%%%%%%%%%%%%%%%%%%%%%%%%
%%%%%%%%%%%%%%%%%%%%%%%%%%%%%%%%%%%%%%%%%
%%%%%%%%%%%   			model.A 				%%%%%%%%%%%
%%%%%%%%%%%%%%%%%%%%%%%%%%%%%%%%%%%%%%%%%
%%%%%%%%%%%%%%%%%%%%%%%%%%%%%%%%%%%%%%%%%

\bigskip
\item[\textbf{(Sph.A8)}] $G / K = \SL(2n+1) / \Sp(2n)$ 
$$\widehat G / \widehat K = \SL(2n+2) / \Sp(2n+2)$$ 
$$H / K = \GL(2n) / \Sp(2n)$$
%
%\begin{table}[H]
%\begin{tabular}{|c|c|c|}
%    \hline
%%%%%%%%   Riga 1
%$\GL(2n)/\Sp(2n)$ &
%$\SL(2n+1)/\Sp(2n)$ & 
%$\SL(2n+2)/\Sp(2n+2)$ \\
%%
%\hline
%%
%\multirow{2}{*}{\begin{picture}(11400,600)(-400,-200)
%\put(0,0){\usebox{\dthree}}
%\put(3600,0){\usebox{\susp}}
%\put(7200,0){\usebox{\dthree}}
%\end{picture}} 
%&
%\multirow{2}{*}{\begin{picture}(8000,600)(-400,-200)
%\put(0,0){\usebox{\atwoseq}}
%\end{picture}} 
%& 
%\multirow{2}{*}{\begin{picture}(11400,600)(-400,-200)
%\put(0,0){\usebox{\dthree}}
%\put(3600,0){\usebox{\susp}}
%\put(7200,0){\usebox{\dthree}}
%\end{picture}} 
%\\
%& & \\
%\hline
%\end{tabular}
%\end{table}
%
$$
\begin{array}{|c|c|c|}
\hline
\phantom{\Big|} & \text{normalized spherical roots} & \text{fund. spherical weights}\\
\hline \hline
\phantom{\Big|} G/K \phantom{\Big|} & \alpha_i + \alpha_{i+1}\ (i=1,\ldots,2n-1) & \omega_i\ (i=1,\ldots,2n) \\
\hline
\phantom{\Big|}\widehat G / \widehat K\phantom{\Big|} & \alpha'_{2i-1} + 2\alpha'_{2i} + \alpha'_{2i+1}\ (i=1,\ldots,n) & \omega'_{2i}\ (i=1,\ldots, n) \\
\hline
\phantom{\Big|}H/K \phantom{\Big|} & \alpha''_{2i-1} + 2\alpha''_{2i} + \alpha''_{2i+1}\ (i=1,\ldots, n-1) & \omega''_{2i}\ (i=1,\ldots,n-1),\ \pm\omega''_{2n} \\
\hline
\end{array}
$$

$$
\begin{array}{|ccccl|}
\hline
\phantom{\Big|}\hat \varphi & : & \sum_{i=1}^{2n} a_i \omega_i  & \longmapsto & \sum_{i=1}^n (a_{2i-1} + a_{2i}) \omega'_{2i}  \\
\hline
\phantom{\Big|}\bar \varphi & : & \sum_{i=1}^{2n} a_i \omega_i & \longmapsto & \sum_{i=1}^{n-1} (a_{2i}+a_{2i+1}) \omega''_{2i} + (a_{2n} - \sum_{i=1}^n a_{2i-1}) \omega''_{2n} \\
\hline
\end{array}
$$

%$$
%\begin{array}{ccccc}
%\varphi_{G_0} & : & \sum_{i=1}^{2n} a_i \omega_i^{\sfA_{2n}} & \longmapsto & \sum_{i=1}^{n-1} (a_{2i}+a_{2i+1}) \omega_{2i}^{\sfA_{2n-1}} + 2n \sum_{i=1}^n a_{2i-1} \omega_{2n} - \sum_{i=1}^n 2i (a_{2i-1}+a_{2i}) \omega_{2n}
%\end{array}
%$$

%%%%%%%%%%%%%%%%%%%%%%%%%%%%%%%%%%%%%%%%%
%%%%%%%%%%%%%%%%%%%%%%%%%%%%%%%%%%%%%%%%%
%%%%%%%%%%%%   			Levi.C 				%%%%%%%%%%%
%%%%%%%%%%%%%%%%%%%%%%%%%%%%%%%%%%%%%%%%%
%%%%%%%%%%%%%%%%%%%%%%%%%%%%%%%%%%%%%%%%%

\bigskip
\item[\textbf{(Sph.A10)}] $G / K = \Sp(2n) / [\GL(1) \times \Sp(2n-2)]$ 
$$\widehat G / \widehat K = \SL(2n) / \GL(2n-1)$$ 
$$H / K = \Sp(2) / \GL(1)$$
%	
%	\begin{table}[H]
%\begin{tabular}{|c|c|c|}
%    \hline
%%%%%%%%   Riga 1
%$[\Sp(2)/\GL(1)] \times [\Sp(2n-2)/\Sp(2n-2)]$ & $\Sp(2n)/[\GL(1) \times  \Sp(2n-2)]$ & $\SL(2n)/\GL(2n-1)$ \\
%%
%\hline
%%
%\multirow{3}{*}{\begin{picture}(7000,600)(-500,-200)
% \put(0,0){\usebox{\aone}}
%\put(1800,0){\usebox{\susp}}
%\put(5400,0){\usebox{\leftbiedge}}
%\end{picture}} 
%& 
%\multirow{3}{*}{\begin{picture}(7000,600)(900,-200)
%\put(0,0){\usebox{\shortcm}}
%\put(0,0){\usebox{\aone}}
%\end{picture}} 
%& 
%\multirow{3}{*}{\begin{picture}
%(6000,800)(-1400,-200)
%\put(0,0){\usebox{\shortam}}
%\end{picture}} \\
%& & \\
%& & \\
%\hline
%\end{tabular}
%\end{table}
%
$$
\begin{array}{|c|c|c|}
\hline
\phantom{\Big|} & \text{normalized spherical roots} & \text{fund. sph. weights}\\
\hline \hline
\phantom{\Big|}G/K\phantom{\Big|} & \alpha_1, \ \alpha_1 + 2(\alpha_2 + \ldots + \alpha_{n-1}) + \alpha_n & 2\omega_1, \ \omega_2 \\
\hline
\phantom{\Big|}\widehat G / \widehat K\phantom{\Big|} & \alpha'_1 + \ldots + \alpha'_{2n-1} & \omega'_1 + \omega'_{2n-1} \\
\hline
\phantom{\Big|}H/K \phantom{\Big|} &  \alpha''_1 & 2\omega''_1\\
\hline
\end{array}
$$

$$
\begin{array}{|ccccl|}
\hline
\phantom{\Big|} \hat \varphi & : & 2a \omega_1+ b \omega_2 & \longmapsto & (a+b) (\omega'_1 + \omega'_{2n-1}) \\
\hline
\phantom{\Big|} \bar \varphi & : & 2a \omega_1 + b \omega_2 &\longmapsto & 2a \omega''_1 \\
\hline
\end{array}
$$

%%%%%%%%%%%%%%%%%%%%%%%%%%%%%%%%%%%%%%%%%
%%%%%%%%%%%%%%%%%%%%%%%%%%%%%%%%%%%%%%%%%
%%%%%%%%%%%   			triality.B				%%%%%%%%%%
%%%%%%%%%%%%%%%%%%%%%%%%%%%%%%%%%%%%%%%%%
%%%%%%%%%%%%%%%%%%%%%%%%%%%%%%%%%%%%%%%%%

\bigskip
\item[\textbf{(Sph.A13)}] $G / K = \Spin(9) / \Spin(7)$ 
$$\widehat G / \widehat K = \SO(16) / \SO(15)$$ 
$$H/K = \Spin(8) / \Spin(7)$$
%
%\begin{table}[H]
%\begin{tabular}{|c|c|c|}
%    \hline
%%%%%%%%   Riga 1
%$\Spin(8)/\Spin(7)$ &
%$\Spin(9)/\Spin(7)$ & 
%$\SO(16)/\SO(15)$ \\
%%
%\hline
%%
%\multirow{3}{*}{\begin{picture}
%(3800,800)(-300,-200)
%\put(0,0){\usebox{\dynkindfour}}
%\put(3000,-1200){\usebox{\gcircle}}
%\put(2700,-780){\tiny1\!/2}
%\end{picture}}
%&
%\multirow{3}{*}{\begin{picture}(7000,600)(-800,-200)
%\put(0,0){\usebox{\dynkinbfour}}
%\put(0,0){\usebox{\gcircle}}
%\put(5400,0){\usebox{\gcircle}}
%\end{picture}} 
%& 
%\multirow{3}{*}{\begin{picture}
%(12000,800)(-800,-200)
%\put(0,0){\usebox{\dynkindeight}}
%\put(0,0){\usebox{\gcircle}}
%\put(-500,420){\tiny1\!/2}
%\end{picture}} \\
%& & \\
%& & \\
%\hline
%\end{tabular}
%\end{table}
%
$$
\begin{array}{|c|c|c|}
\hline
\phantom{\Big|} & \text{normalized spherical roots} & \text{fund. sph. weights}\\
\hline \hline
\phantom{\Big|}G/K \phantom{\Big|} & \alpha_1+\alpha_2+\alpha_3+\alpha_4, \ \alpha_2 + 2\alpha_3 + 3\alpha_4 & \omega_1, \ \omega_4 \\
\hline
\phantom{\Big|}\widehat G / \widehat K\phantom{\Big|} & 2(\alpha'_1 + \ldots + \alpha'_6) + \alpha'_7+\alpha'_8 & \omega'_1 \\
\hline
\phantom{\Big|} H/K \phantom{\Big|} & \alpha''_1+2\alpha''_2+\alpha''_3+2\alpha''_4 & \omega''_4\\
\hline
\end{array}
$$

$$
\begin{array}{|ccccl|}
\hline
\phantom{\Big|} \hat \varphi & : & a \omega_1 + b \omega_4 & \longrightarrow & (2a+b) \omega'_1 \\
\hline
\phantom{\Big|} \bar \varphi & : & a \omega_1 + b \omega_4 & \longmapsto & b \omega''_4 \\
\hline
\end{array}
$$

%%%%%%%%%%%%%%%%%%%%%%%%%%%%%%%%%%%%%%%%%
%%%%%%%%%%%%%%%%%%%%%%%%%%%%%%%%%%%%%%%%%
%%%%%%%%%%%%   			triality 				%%%%%%%%%%%
%%%%%%%%%%%%%%%%%%%%%%%%%%%%%%%%%%%%%%%%%
%%%%%%%%%%%%%%%%%%%%%%%%%%%%%%%%%%%%%%%%%

\bigskip
\item[\textbf{(Sph.A14)}] $G / K = \Spin(8) / \sfG_2$ 
$$\widehat G / \widehat K = \Spin(8)/\Spin(7) \times \Spin(8)/\Spin(7)$$ 
(where the two copies of $\Spin(7)$ are not conjugated in $\Spin(8)$)
$$H/K = \Spin(7)/ \sfG_2$$
%
%\begin{table}[H]
%\begin{tabular}{|c|c|c|}
%    \hline
%%%%%%%%   Riga 1
%$\Spin(7)/\sfG_2$ &
%$\SO(8)/\sfG_2$ & 
%$\SO(8)/\SO(7) \; \times \; \SO(8) / \Spin(7)$ \\
%%
%\hline
%%
%\multirow{4}{*}{\begin{picture}(7000,600)(-1600,-200)
%\put(0,0){\usebox{\dynkinbthree}}
%\put(3600,0){\usebox{\gcircle}}
%\put(3100,420){\tiny1\!/2}
%\end{picture}} 
%&
%\multirow{4}{*}{\begin{picture}(5800,800)(-1400,-200)
%\put(0,0){\usebox{\athreene}}
%\put(0,0){\usebox{\athreese}}
%\put(1800,0){\usebox{\athreebifurc}}
%\end{picture}}
%& 
%\multirow{4}{*}{\begin{picture}
%(3800,800)(1000,-200)
%\put(0,0){\usebox{\dynkindfour}}
%\put(-500,420){\tiny1\!/2}
%\put(0,0){\usebox{\gcircle}}
%\end{picture}}
%\multirow{4}{*}{\begin{picture}
%(3800,800)(-1000,-200)
%\put(0,0){\usebox{\dynkindfour}}
%\put(3000,1200){\usebox{\gcircle}}
%\put(2500,1620){\tiny1\!/2}
%\end{picture}} \\
%& & \\
%& & \\
%& & \\
%\hline
%\end{tabular}
%\end{table}

$$
\begin{array}{|c|c|c|}
\hline
\phantom{\Big|} & \text{normalized spherical roots} & \text{fund. sph. weights}\\
\hline \hline
\phantom{\Big|}G/K \phantom{\Big|} & \alpha_1+\alpha_2+\alpha_3, \ \alpha_1+\alpha_2+\alpha_4, \ \alpha_2 + \alpha_3 + \alpha_4 & \omega_1, \  \omega_3, \ \omega_4 \\
\hline
\phantom{\Big|}\widehat G / \widehat K\phantom{\Big|} & 2\alpha'_1 + 2\alpha'_2 + \alpha'_3 + \alpha'_4, \ \alpha''_1 + 2\alpha''_2+ 2\alpha''_3 + \alpha''_4 & \omega'_1, \ \omega''_3 \\
\hline
\phantom{\Big|} H/K \phantom{\Big|} &  \alpha'''_1+2\alpha'''_2+3\alpha'''_3 & \omega'''_3\\
\hline
\end{array}
$$

$$
\begin{array}{|ccccl|}
\hline
\phantom{\Big|} \hat \varphi & : & a \omega_1 + b \omega_3 + c \omega_4 & \longmapsto & (a+c)\omega'_1 + (b+c)\omega''_3 \\
\hline
\phantom{\Big|} \bar \varphi & : & a \omega_1 + b \omega_3 + c \omega_4 & \longmapsto &  (a+b)\omega'''_3 \\
\hline
\end{array}
$$

\end{itemize}

\bigskip

Given $\lambda \in \Lambda_{G/K}^+$, recall our notation $\hat \lambda=\hat \varphi(\lambda)$ and $\bar \lambda=\bar \varphi(\lambda)$.  If $\mu \in \Lambda^+_{G/K}$, we write $\mu \sim \lambda$ if $\hat \mu =\hat \lambda$. Identifying $\CC[\widehat G/ \widehat K] \simeq \CC[G/K]$, by definition we have then
$$
	E_{\widehat G/\widehat K}(\hat \lambda) = \bigoplus_{\lambda' \sim \lambda} E_{G/K}(\lambda') .
$$

\begin{theorem}	\label{teo:decomposizione}
Let $(G,K)$ be one of the following spherical pairs of Table  \ref{tab:sph_pairs_A}
$$ \mathrm{Sph.A6, \ Sph.A8, \ Sph.A10, \ Sph.A13, \ Sph.A14. }$$
Then for all $\lambda,\mu,\nu\in\Lambda^+_{G/K}$ we have
$$
\begin{array}{ccc}
E_{G/K}(\nu) \subset E_{G/K}(\lambda) \cdot E_{G/K}(\mu) &
\Longleftrightarrow &
\left\{
\begin{array}{c}
	E_{\widehat G/\widehat K}\big(\hat\nu\big) \subset E_{\widehat G/\widehat K}\big(\hat\lambda\big) \cdot E_{\widehat G/\widehat K}\big(\hat\mu\big) \phantom{\Big|} \\ 
	E_{H/K}(\bar \nu) \subset E_{H/K}(\bar \lambda) \cdot E_{H/K}(\bar\mu) \end{array}\right.
\end{array}
$$
%\footnote{Nei casi (Sph.A10) e (modB) mi sembra sensato che si generalizzi a pesi quasi sferici.} 
\end{theorem}

\begin{proof}
Suppose that $E_{G/K}(\nu) \subset E_{G/K}(\lambda) \cdot E_{G/K}(\mu)$. Looking at the associated $\widehat G$-modules generated in the coordinate ring $\CC[\widehat G/\widehat K]\ \simeq \CC[G/K]$ it follows that $E_{\widehat G/\widehat K}\big(\hat \nu\big)  \subset E_{\widehat G/\widehat K}\big(\hat \lambda\big) \cdot E_{\widehat G/\widehat K}\big(\hat \mu\big)$ as well. Similarly, looking at the images in $\CC[H/K]$ we see that $E_{H/K}(\bar \nu)  \subset E_{H/K}(\bar \lambda) \cdot E_{H/K}(\bar \mu)$.

Suppose conversely that both the inclusions
	$E_{\widehat G/\widehat K}(\hat \nu)  \subset E_{\widehat G/\widehat K}(\hat \lambda) \cdot E_{\widehat G/\widehat K}(\hat \mu)$ and $E_{H/K}(\bar \nu)  \subset E_{H/K}(\bar \lambda) \cdot E_{H/K}(\bar \mu)$ hold. 	By the first inclusion, the restriction of functions from $G/K$ to $H/K$ induces a commutative diagram as follows
\[\begin{tikzcd}        
 \bigoplus_{\nu' \sim \nu } E_{G/K}(\nu') 
\arrow[hook]{rr}\arrow[two heads]{d}    
&& 
 \big(\bigoplus_{\lambda' \sim \lambda} E_{G/K}(\lambda') \big)\cdot  \big(\bigoplus_{\mu' \sim \mu } E_{G/K}(\mu') \big)  
\arrow[two heads]{d}  
\\
      \bigoplus_{\nu' \sim \nu } E_{H/K}(\bar \nu') 
\arrow[hook]{rr} 
& &
\big(\bigoplus_{\lambda' \sim \lambda} E_{H/K}(\bar \lambda') \big)\cdot  \big(\bigoplus_{\mu' \sim \mu } E_{H/K}(\bar \mu') \big)\end{tikzcd}\]

%\bigoplus_{
%      \begin{subarray}{c}
%      \lambda' \sim \lambda, \, \mu' \sim \mu\\
%       V_G(\nu') \subset V_G(\lambda') \cdot V_G(\mu')
%       \end{subarray}} 
%       V_{G_0}(\bar{\nu'})
	
	By the second inclusion, there are $\lambda' \sim \lambda$, $\mu'\sim \mu$ and $\nu' \sim \nu$ such that $\bar \lambda' = \bar \lambda$, $\bar \mu' = \bar \mu$, $\bar \nu' = \bar \nu$ and $E_{G/K}(\nu') \subset E_{G/K}(\lambda') \cdot E_{G/K}(\mu')$. On the other hand by Corollary \ref{cor:injectivity}, the restriction of $\bar \varphi$ to the fibers of $\hat \varphi$ is injective. Thus $\lambda' = \lambda$, $\mu' = \mu$, $\nu' = \nu$, and we get $E_{G/K}(\nu) \subset E_{G/K}(\lambda) \cdot E_{G/K}(\mu)$.	
\end{proof}

\begin{corollary}	\label{cor:decomposizione}
Let $(G,K)$ be one of the following spherical pairs of Table \ref{tab:sph_pairs_A}
$$ \mathrm{Sph.A6, \ Sph.A8, \ Sph.A10, \ Sph.A13, \ Sph.A14. }$$
If Conjecture \ref{conj:stanley_cor} holds true, then Conjecture \ref{conj:affine_spherical_A} holds true for $(G,K)$. In particular, Conjecture \ref{conj:affine_spherical_A} holds true in the cases
$$ \mathrm{Sph.A6, \ Sph.A10, \ Sph.A13, \ Sph.A14. }$$
\end{corollary}

\begin{proof}
%In all cases, $\widehat G/\widehat K$ and $H/K$ are symmetric varieties with restricted root system of type $\sfA$. Thus we are in the hypotheses of Conjecture \ref{conj:stanley_cor}, and when both $\Phi_{\widehat G/\widehat K}$ and $\Phi_{H/K}$ have rank one we can apply Stanley's Pieri rule for Jack symmetric functions.
%
Set $X = G/K$, $Y = \widehat G/\widehat K$ and $Z = H/K$. Then by Proposition \ref{prop:isogeny} we have an isogeny of reductive groups
$$
	G_Y \times G_Z \rightarrow G_X
$$
Notice that a $G_X$-module is irreducible if and only if it is irreducible as a $(G_Y \times G_Z)$-module under the previous map.

If $\lambda \in \Lambda_X^+$, regarding $V_X(\lambda)$ as a $(G_Y \times G_Z)$-module, by the definition of $\hat \lambda$ and $\bar \lambda$ we see that
$$
V_X(\lambda) \simeq V_Y \big(\hat \lambda\big) \boxtimes V_Z\big(\bar \lambda\big).
$$ 
In particular, given $\lambda, \mu , \nu \in \Lambda^+_X$, we have
$$
\begin{array}{ccc}
	V_X(\nu) \subset V_X(\lambda) \otimes V_X(\mu) \Longleftrightarrow
\left\{
\begin{array}{c}
	V_Y\big(\hat\nu\big) \subset V_Y \big(\hat\lambda\big) \otimes V_Y\big(\hat\mu\big) \phantom{\Big|} \\ 
	V_Z(\bar \nu) \subset V_Z(\bar \lambda) \otimes V_Z(\bar\mu) \end{array}\right.
\end{array}
$$
Thus by Theorem \ref{teo:decomposizione}, we see that Conjecture \ref{conj:affine_spherical_A} holds for $X$ if and only if it holds both for $Y$ and $Z$.

If both $Y$ and $Z$ are direct products of symmetric varieties with restricted root system of type $\sfA_1$, then by Stanley's Pieri rule for Jack symmetric functions we see that Conjecture \ref{conj:stanley_cor} holds true both for $Y$ and $Z$. Thus in this case the claim follows for $X$ as well.
\end{proof}

\begin{remark} \label{oss:sfera_transitiva_nonsemplice}
There is another case which can be put in the framewok outlined at the beginning of this subsection, even though $G$ is not simple. In this case
$$G/K = [\Sp(2) \times \Sp(2n)] / [\diag(\Sp(2)) \times \Sp(2n-2)]$$
and $\Phi_{G/K}$ is of type $\sfA_1 \times \sf A_1$, and Conjecture \ref{conj:affine_spherical_A} holds true thanks to the same arguments of Theorem \ref{teo:decomposizione} and Corollary \ref{cor:decomposizione}. The corresponding symmetric varieties are
$$
	\widehat G/\widehat K = \SO(4n)/\SO(4n-1)
$$
\begin{align*}
	H/K = [\Sp(2) \times \Sp(2) \times \Sp(2n-2)]/[\diag(\Sp(2)) \times \Sp(2n-2)]  & \\
	 \simeq [\Sp(2) \times \Sp(2) ]/\diag(\Sp(2))& 
\end{align*}
which have both restricted root system of type $\sfA_1$. The associated spherical roots and spherical weights are described in the tables here below, together with the maps $\hat \varphi$ and $\bar \varphi$.

$$
\begin{array}{|c|c|c|}
\hline
\phantom{\Big|} & \text{normalized spherical roots} & \text{fund. spherical weights}\\
\hline \hline
\phantom{\Big|} G/K & \alpha_0 + \alpha_1, \, \alpha_1 + 2(\alpha_2 + \ldots + \alpha_{n-1}) + \alpha_n & \omega_0+\omega_1, \ \omega_2 \\
\hline
\phantom{\Big|} \widehat G/\widehat K & 2(\alpha'_1 + \ldots + \alpha'_{2n-2}) + \alpha'_{2n-1} + \alpha'_{2n} & \omega_1' \\
\hline
\phantom{\Big|} H/K & \alpha_0'' + \alpha_1'' & \omega_0'' + \omega_1'' \\
\hline
\end{array}
$$

$$
\begin{array}{|ccccl|}
\hline
\phantom{\Big|} \hat \varphi & : & a (\omega_0+\omega_1) + b\omega_2 & \longmapsto & (a+2b)\omega'_1  \\
\hline
\phantom{\Big|} \bar \varphi & : & a (\omega_0+\omega_1) + b\omega_2 & \longmapsto  & a (\omega_0''+\omega_1'') \\
\hline
\end{array}
$$

%\begin{table}[H]
%\begin{tabular}{|c|c|c|}
%    \hline
%$[\SL(2) \times \SL(2)]/\diag(\SL(2)) $ & $[\SL(2) \times \Sp(2n+2)]/[\diag(\SL(2)) \times \Sp(2n)]$ & $\SO(4n+4)/\SO(4n+3)$
% \\
%    \hline \hline
%%%%%%%%   Riga 1
%\multirow{3}{*}{
%\begin{picture}(3300,1800)(-300,-900)
%\multiput(0,0)(2700,0){2}{\usebox{\vertex}}
%\multiput(0,0)(2700,0){2}{\usebox{\wcircle}}
%\multiput(0,-300)(2700,0){2}{\line(0,-1){600}}
%\put(0,-900){\line(1,0){2700}}
%\end{picture}} 
%& \multirow{3}{*}{\begin{picture}(9000,600)(900,-200)
%\multiput(0,0)(2700,0){2}{\usebox{\vertex}}
%\multiput(0,0)(2700,0){2}{\usebox{\wcircle}}
%\multiput(0,-300)(2700,0){2}{\line(0,-1){600}}
%\put(0,-900){\line(1,0){2700}}
%\put(2700,0){\usebox{\shortcm}}
%\end{picture}}  &
%\multirow{3}{*}{
%\begin{picture}(6900,2400)(-300,-600)
%\put(0,0){\usebox{\shortdm}}
%\put(-500,420){\tiny1\!/2}
%\end{picture}}  \\
%&  & \\ & & \\ & & \\
%    \hline
%\end{tabular}
%\end{table}

$ $\\
\end{remark}

%%%%%%%%%%%%%%%%%%%%%%%%%%%%%%%%%%%%%%%
%%%%%%%%%%%%%%%%%%%%%%%%%%%%%%%%%%%%%%%
\subsection{Other cases with restricted root system of type $\sfA$.}
\label{ssec:other cases}
%%%%%%%%%%%%%%%%%%%%%%%%%%%%%%%%%%%%%%%
%%%%%%%%%%%%%%%%%%%%%%%%%%%%%%%%%%%%%%%

We now explain how the cases Sph.A7, Sph.A9, Sph.A11 of Table \ref{tab:sph_pairs_A} follow (almost immediately) from their companion cases Sph.A6, Sph.A8, Sph.A10.\\

\begin{itemize}
\item[\textbf{(Sph.A7)}] $G = \SL(n), \ K = \GL(n-1)$. This case is immediately reduced to Sph.A6. Denote indeed $K' = \SL(n-1)$: then $K = \rmN_G(K')$, and the spherical roots of $G/K$ are the same as those of $G/K'$. On the other hand $\Omega_{G/K} = \Lambda_{G/K'}$, and $V(\lambda)^{(K)} = V(\lambda)^{K'}$ for all $\lambda \in \Omega_{G/K}^+$, therefore $\mathbb C[G]^{(K)} = \mathbb C[G]^{K'}$. \\

\item[\textbf{(Sph.A9)}] $G = \SL(2n+1), \ K = \GL(1) \times \Sp(2n)$. This case is immediately reduced to Sph.A8. Consider indeed $K' = \Sp(2n)$: then $K = \rmN_G(K')$, and the spherical roots of $G/K$ are the same as those of $G/K'$. Moreover, as in the previous case, we have $\mathbb C[G]^{(K)} = \mathbb C[G]^{K'}$.\\

\item[\textbf{(Sph.A11)}] $G = \Sp(2n), \ K = \Sp(2) \times \Sp(2n-2)$. This case is immediately reduced to Sph.A10. Set indeed $K' = \GL(1) \times \Sp(2n-2)$: then $K' \subset K$, thus we have an inclusion $\mathbb C[G]^K \subset \mathbb C[G]^{K'}$. On the other hand the unique spherical root of $G/K$ is a spherical root of $G/K'$ as well, and if $\lambda, \mu \in \Lambda_{G/K}$, then $\mu \leq_{G/K} \lambda$ if and only if $\mu \leq_{G/K'} \lambda$.

%	\begin{table}[H]
%\begin{tabular}{|c|c|}
%    \hline
%%%%%%%%   Riga 1
%$\Sp(2n)/[\Sp(2) \times \Sp(2n-2)]$ & $\Sp(2n)/[\GL(1) \times \Sp(2n-2)]$ \\
%%
%\hline
%%
%\multirow{3}{*}{\begin{picture}(7000,600)(900,-200)
%\put(0,0){\usebox{\shortcm}}
%\end{picture}} 
%& 
%\multirow{3}{*}{\begin{picture}(7000,600)(900,-200)
%\put(0,0){\usebox{\shortcm}}
%\put(0,0){\usebox{\aone}}
%\end{picture}} 
%\\
%& \\
%& \\
%\hline
%\end{tabular}
%\end{table}

\begin{table}[H]
\begin{tabular}{|c|c|c|}
\hline
$\phantom{\Big|}$ & \text{normalized spherical roots} & \text{fund. sph. weights}\\
\hline \hline
$\phantom{\Big|}G/K\phantom{\Big|}$ & $\alpha_1 + 2(\alpha_2 + \ldots + \alpha_{n-1}) + \alpha_n$ & $\omega_2$ \\
\hline
$\phantom{\Big|}G/K'\phantom{\Big|}$ & $\alpha_1, \ \alpha_1 + 2(\alpha_2 + \ldots + \alpha_{n-1}) + \alpha_n$ & $2\omega_1, \ \omega_2$ \\
\hline
\end{tabular}

\end{table}
\end{itemize}

As a consequence of the previous analysis, we get the following.

\begin{proposition}\label{prop:altri}
Let $(G,K)$ be one of the spherical pairs Sph.A7, Sph.A9, Sph.A11 of Table \ref{tab:sph_pairs_A}. 
If Conjecture \ref{conj:stanley_cor} holds true, then Conjecture \ref{conj:affine_spherical_A} holds true for $(G,K)$. In particular, Conjecture \ref{conj:affine_spherical_A} holds true for Sph.A7 and Sph.A11.
\end{proposition}

%%%%%%%%%%%%%%%%%%%%%%%%%%%%%%%%%%%%%%%
%%%%%%%%%%%%%%%%%%%%%%%%%%%%%%%%%%%%%%%
\section{A conjectural rule in the case of spherical Levi subgroups}
\label{sec:Levi}
%%%%%%%%%%%%%%%%%%%%%%%%%%%%%%%%%%%%%%%
%%%%%%%%%%%%%%%%%%%%%%%%%%%%%%%%%%%%%%%

There is another case where we can apply the techniques of Section \ref{sec:casi_tipo_A}, and find two morphisms which decompose the root system of an affine spherical variety into those of two symmetric varieties. However in this case the root system generated by the spherical roots is of type $\sfB$, and as far as we know there is no rule (not even conjectural) to multiply spherical functions on a symmetric variety in this case.

\subsection{The spherical pair $(\SO(2p+1),\GL(p))$.} \label{ssec:modelloB}
Set $G = \SO(2p+1)$ and $K = \GL(p)$ with $p\geq3$. Then the homogeneous space $G/K$ is spherical, and it is a model space for $G$ (see \cite{GZ1}, \cite{GZ2}, \cite{Lu2}).

The spherical root system $\Phi_{G/K}$ is of type $\sfB_{\lfloor p/2\rfloor} \times \sfB_{\lceil p/2\rceil}$ if $p\geq4$, and of type $\sfA_1\times\sfB_2$ if $p=3$. In this case as well, it is possible  to define reductive groups $H \subset G \subset \widehat G$ and two symmetric varieties $H/K$ and $\widehat G/\widehat K$ endowed with equivariant morphisms
$$
	H/K \hookrightarrow G/K \stackrel{\sim}{\longrightarrow} \widehat G/\widehat K
$$
which produce an isogeny of root data such as in Proposition \ref{prop:isogeny}. This will allow to extend Theorem \ref{teo:decomposizione} to this case as well, and to reduce the multiplication of the spherical functions on $G/K$ to the case of two symmetric varieties.

More precisely, define
$$\widehat G = \SO(2p+2), \quad \widehat K = \GL(p+1), \quad H = \SO(2p).$$
The root systems $\Phi_{\widehat G/\widehat K}$ and $\Phi_{H/K}$ are of type $\sfB$, whereas the restricted root systems $\widetilde \Phi_{\widehat G/\widehat K}$ and $\widetilde \Phi_{H/K}$ are either of tye $\sfB\sfC$ or of type $\sfC$, depending on the parity of $p$.
 
As in Section \ref{sec:casi_tipo_A}, the equivariant morphisms $H/K \hookrightarrow G/K \stackrel{\sim}{\to} \widehat G/\widehat K$ induce linear maps of the respective weight monoids
$$
	\hat \varphi : \Lambda^+_{G/K} \longrightarrow \Lambda^+_{\widehat G/\widehat K} , \qquad \qquad
	\bar \varphi : \Lambda^+_{G/K} \longrightarrow \Lambda^+_{H/K} .
$$
For $\lambda \in \Lambda_{G/K}^+$, we set $\hat \lambda = \hat \varphi(\lambda)$ and $\bar \lambda = \bar \varphi(\lambda)$. 

The spherical roots and the weight lattices, as well as the maps $\hat \varphi$ and $\bar \varphi$, are described in the tables below. We separate the two cases $p$ even and $p$ odd.

\bigskip
\noindent\textbf{($p=2n$ even)}
%
%\begin{table}[H]
%\begin{tabular}{|c|c|c|}
%    \hline
%%%%%%%%   Riga 1
%$\SO(4n)/\GL(2n)$ &
%$\SO(4n+1)/\GL(2n)$ & 
%$\SO(4n+2)/\GL(2n+1)$ \\
%%
%\hline
%%
%\multirow{4}{*}{\begin{picture}(14000,600)(-200,-400)
%\put(0,0){\usebox{\dthree}}
%\put(3600,0){\usebox{\susp}}
%\put(7200,0){\usebox{\dthree}}
%\put(10800,0){\usebox{\dynkindfour}}
%\put(12600,0){\usebox{\gcircle}}
%\put(13800,-1200){\usebox{\aone}}
%\put(13800,-600){\usebox{\tonw}}
%\end{picture}} 
%&
%\multirow{3}{*}{\begin{picture}(9400,600)(-200,200)
%\put(0,0){\usebox{\atwoseq}}
%\put(7200,0){\usebox{\btwo}}
%\put(9000,0){\usebox{\aone}}
%\put(9000,600){\usebox{\tow}}
%\end{picture}} 
%& 
%\multirow{3}{*}{\begin{picture}(12600,600)(-200,200)
%\put(0,0){\usebox{\dthree}}
%\put(3600,0){\usebox{\susp}}
%\put(7200,0){\usebox{\dthree}}
%\put(10800,0){\usebox{\athreebifurc}}
%\end{picture}} 
%\\
%& & \\
%& & \\
%& & \\
%\hline
%\end{tabular}
%\end{table}
%
$$
\begin{array}{|c|c|c|}
\hline
\phantom{\Big|} & \text{normalized spherical roots} & \text{fund. sph. weights}\\
\hline \hline
\phantom{\Big|}G/K \phantom{\Big|} & 
\begin{array}{l}
\alpha_i+\alpha_{i+1} \ (i=1, \ldots, 2n-1),\ \alpha_{2n} 
\end{array} 
& 
\begin{array}{l}\omega_i\ (i=1, \ldots ,2n-1),\\ 2\omega_{2n} \end{array}\\
\hline
\phantom{\Big|}\widehat G / \widehat K\phantom{\Big|} & 
\begin{array}{l}
\alpha'_{2i-1} + 2\alpha'_{2i}+\alpha'_{2i+1}\ (i=1, \ldots, n-1) , \\
\alpha'_{2n-1} + \alpha'_{2n} + \alpha'_{2n+1} 
\end{array}
 & \begin{array}{l}\omega'_{2i}\ (i=1, \ldots ,n-1), \\ \omega'_{2n} + \omega'_{2n+1} \end{array}\\
\hline
\phantom{\Big|}H/K\phantom{\Big|} & 
\begin{array}{l}
\alpha''_{2i-1} + 2\alpha''_{2i}+\alpha''_{2i+1} \ (i=1, \ldots, n-1),  \\
\alpha''_{2n}  
\end{array}
 & \begin{array}{l}\omega''_{2i}\ (i=1, \ldots ,n-1), \\ 2\omega''_{2n} \end{array}\\
\hline
\end{array}
$$

$$
\begin{array}{|ccccl|}
\hline
\phantom{\Big|}\hat \varphi & : & \sum_{i=1}^{2n} a_i \omega_i & \longmapsto & \big(\sum_{i=1}^{n-1} (a_{2i-1} + a_{2i}) \omega'_{2i}\big) + (a_{2n-1}+\frac{a_{2n}}{2}) (\omega'_{2n} + \omega'_{2n+1})\\
\hline
\phantom{\Big|}\bar \varphi & : & \sum_{i=1}^{2n} a_i \omega_i & \longmapsto &  \big(\sum_{i=1}^{n-1} (a_{2i}+a_{2i+1}) \omega''_{2i}\big) + a_{2n} \omega''_{2n}\\
\hline
 \end{array}
$$

\bigskip
\noindent \textbf{($p=2n+1$ odd)}
%
%\begin{table}[H]
%\begin{tabular}{|c|c|c|}
%    \hline
%%%%%%%%   Riga 1
%$\SO(4n+2)/\GL(2n+1)$ &
%$\SO(4n+3)/\GL(2n+1)$ & 
%$\SO(4n+4)/\GL(2n+2)$ \\
%%
%\hline
%%
%\multirow{3}{*}{\begin{picture}(12600,600)(-200,200)
%\put(0,0){\usebox{\dthree}}
%\put(3600,0){\usebox{\susp}}
%\put(7200,0){\usebox{\dthree}}
%\put(10800,0){\usebox{\athreebifurc}}
%\end{picture}} 
%&
%\multirow{3}{*}{\begin{picture}(9400,600)(-200,200)
%\put(0,0){\usebox{\atwoseq}}
%\put(7200,0){\usebox{\btwo}}
%\put(9000,0){\usebox{\aone}}
%\put(9000,600){\usebox{\tow}}
%\end{picture}} 
%& 
%\multirow{4}{*}{\begin{picture}(14000,600)(-200,-400)
%\put(0,0){\usebox{\dthree}}
%\put(3600,0){\usebox{\susp}}
%\put(7200,0){\usebox{\dthree}}
%\put(10800,0){\usebox{\dynkindfour}}
%\put(12600,0){\usebox{\gcircle}}
%\put(13800,-1200){\usebox{\aone}}
%\put(13800,-600){\usebox{\tonw}}
%\end{picture}} 
%\\
%& & \\
%& & \\
%& & \\
%\hline
%\end{tabular}
%\end{table}
%
$$
\begin{array}{|c|c|c|}
\hline
\phantom{\Big|} & \text{normalized spherical roots} & \text{fund. sph. weights}\\
\hline \hline
\phantom{\Big|}G/K \phantom{\Big|} & 
\alpha_i+\alpha_{i+1} \ (i=1, \ldots, 2n),\  \alpha_{2n+1} 
& \begin{array}{l}\omega_i \ (i=1,\ldots ,2n), \\ 2\omega_{2n+1}\end{array} \\
\hline
\phantom{\Big|}\widehat G / \widehat K\phantom{\Big|} & 
\begin{array}{l}
\alpha'_{2i-1} + 2\alpha'_{2i}+\alpha'_{2i+1} \ (i=1, \ldots, n), \\
\alpha'_{2n+2}  
\end{array}
 & \begin{array}{l}\omega'_{2i} (i=1, \ldots, n), \\ 2\omega'_{2n+2}\end{array} \\
\hline
\phantom{\Big|}H/K\phantom{\Big|} & 
\begin{array}{l}
\alpha''_{2i-1} + 2\alpha''_{2i}+\alpha''_{2i+1} \ (i=1, \ldots, n-1), \\
\alpha''_{2n-1} + \alpha''_{2n} + \alpha''_{2n+1}  
\end{array}
 & \begin{array}{l}\omega''_{2i} (i=1, \ldots , n-1), \\ \omega''_{2n} + \omega''_{2n+1}\end{array} \\
\hline
\end{array}
$$

$$
\begin{array}{|ccccl|}
\hline
\phantom{\Big|}\hat \varphi & : & \sum_{i=1}^{2n+1} a_i \omega_i & \longmapsto & \big(\sum_{i=1}^n (a_{2i-1} + a_{2i})  \omega'_{2i}\big) + a_{2n+1} \omega'_{2n+2}\\
\hline
\phantom{\Big|}\bar \varphi & : & \sum_{i=1}^{2n+1} a_i \omega_i & \longmapsto & \big(\sum_{i=1}^{n-1} (a_{2i}+a_{2i+1}) \omega''_{2i}\big) + (a_{2n}+\frac{a_{2n+1}}{2}) (\omega''_{2n} + \omega''_{2n+1})\\
\hline
\end{array}
$$

\bigskip
Therefore, in this case as well we obtain the similar results as those in Proposition~\ref{prop:isogeny} and Theorem~\ref{teo:decomposizione}.

%\begin{proposition}\label{prop:isogeny2}
%In the previous notation, the map
%$$
%	\varphi_{G^0} \oplus \varphi_{G_0} : \Lambda_{G/K} \longrightarrow \Lambda_{G^0/K^0} \oplus  \Lambda_{G_0/K_0} 
%$$
%defines an isogeny of based root data
%\[
%(\Xi_{G/K}, \Delta_{G/K}, \Delta_{G/K}^\vee) \longrightarrow
%(\Xi_{G^0/K^0} \oplus \Xi_{G^0/K^0}, \ \Delta_{G^0/K^0} \sqcup \Delta_{G_0/K}, \ \Delta_{G^0/K^0}^\vee \sqcup \Delta_{G_0/K}^\vee )
%\]
%\end{proposition}
%
%Given $\lambda \in \Lambda_{G/K}^+$, denote $\varphi_{G^0}(\lambda) = \lambda^0$ and $\varphi_{G_0}(\lambda) = \lambda_0$. 
%
%\begin{theorem}	\label{teo:decomposizione2}
%We have
%$$
%\begin{array}{ccc}
%E_{G/K}(\nu) \subset E_{G/K}(\lambda) \cdot E_{G/K}(\mu) &
%\Longleftrightarrow &
%\left\{
%\begin{array}{c}
%	E_{G^0/K^0}(\nu^0) \subset E_{G^0/K^0}(\lambda^0) \cdot E_{G^0/K^0}(\mu^0) \phantom{\Big|} \\ 
%	E_{G_0/K}(\nu_0) \subset E_{G_0/K}(\lambda_0) \cdot E_{G_0/K}(\mu_0) \end{array}\right.
%\end{array}
%$$
%\footnote{Nei casi (Sph.A11) e (modB) mi sembra sensato che si generalizzi a pesi quasi sferici.} 
%\end{theorem}

\begin{theorem} \label{teo:isogeny-dec}
Let $(G,K)$ be the spherical pair $(\SO(2p+1),\GL(p))$, then with the above definitions:
	\begin{itemize}
	\item[i)] The map
$$
	(\hat \varphi, \bar \varphi) : \Xi_{G/K} \longrightarrow \Xi_{\widehat G/\widehat K} \oplus  \Xi_{H/K} 
$$
is an isogeny of based root data
\[
\calR_{G/K} \longrightarrow \calR_{\widehat G/\widehat K} \oplus  \calR_{H/K} .
\]
	\item[ii)] For all $\lambda,\mu,\nu\in\Lambda^+_{G/K}$ we have
$$
\begin{array}{ccc}
E_{G/K}(\nu) \subset E_{G/K}(\lambda) \cdot E_{G/K}(\mu) &
\Longleftrightarrow &
\left\{
\begin{array}{c}
	E_{\widehat G/\widehat K}(\hat \nu) \subset E_{\widehat G/\widehat K}(\hat \lambda) \cdot E_{\widehat G/\widehat K}(\hat \mu) \phantom{\Big|} \\ 
	E_{H/K}(\bar \nu) \subset E_{H/K}(\bar \lambda) \cdot E_{H/K}(\bar \mu) \end{array}\right.
\end{array}
$$
%\footnote{Nei casi (Sph.A10) e (modB) mi sembra sensato che si generalizzi a pesi quasi sferici.} 
\end{itemize}
\end{theorem}

The symmetric variety $\SO(2n)/\GL(n)$ however is not covered by Conjecture \ref{conj:stanley_cor}. This prompts us to wonder how the conjecture should be rephrased in this case, and more generally in the case of the symmetric varieties of Hermitian type (that is, the symmetric varieties $G/K$ with $K$ a Levi subgroup). 

Notice that if $(G,K)$ is a spherical pair with $K$ a Levi subgroup of $G$, then either $G/K$ is a symmetric space of Hermitian type, or it is the pair $(\SO(2p+1),\GL(p))$ considered above or it is the pair $(\Sp(2p),\GL(1) \times \Sp(2p-2))$ which appeared as case Sph.A10 in Section \ref{ssec:decomposable cases}. In particular, the analysis in this section reduces the case of the spherical Levi subgroups to the case of the Hermitian symmetric subgroups.

% We now focus to the case of the spherical Levi subgroups. We will consider more generally the case of a reductive spherical subgroup whose normalizer is a Levi subgroup of $G$. Such Levi subgroups are either symmetric subgroups (the symmetric subgroups of Hermitian type), or will be reduced to some symmetric cases in the same way as we did in Subsection \ref{ssec:decomposable cases}.

\subsection{The Hermitian symmetric case.}
Let $X = G/K$ be a symmetric variety of Hermitian type, with $G$ simple and simply connected. This means that the center of $K$ has positive dimension, or equivalently that $K$ is a Levi subgroup of $G$. In this case the restricted root system $\widetilde \Phi_X$ is either of type $\sfC$ or $\sfBC$, whereas the root system $\Phi_X$ is always of type $\sfB$.

For completeness, we list in Table \ref{tab:sym_hermitian} all the symmetric varieties of Hermitian type. In all cases, we also give the simple restricted roots and a minimal set of generators for the weight monoid $\Lambda^+_X$.  

\begin{table}
\caption{Symmetric varieties of Hermitian type}\label{tab:sym_hermitian}
\begin{tabular}{|c|c|c|c|c|c|c|c|}
\hline 
& $\phantom{\bigg|}$ $G$ & $K$ & $\widetilde \Phi_{G/K}$ & $\Phi_{G/K}$ \\
    \hline    \hline
He.1a & $\phantom{\Big|}$ $\SL(p+q)$, $q>p\geq 1$  & $\mathrm S(\GL(p) \times \GL(q))$ & $\sfBC_p$ & $\sfB_p$ \\
    \hline
He.1b & $\phantom{\Big|}$ $\SL(2p)$, $p\geq 1$ &  $\mathrm S(\GL(p) \times \GL(p))$ & $\sfC_p$ & $\sfB_p$ \\
    \hline
He.2 & $\phantom{\Big|}$ $\Spin(n)$, $n\geq5$ &  $\Spin(2) \times \Spin(n-2)$ & $\sfC_2$ & $\sfB_2$ \\
    \hline 
He.3 & $\phantom{\Big|}$ $\Sp(2p)$, $p\geq3$  & $\GL(p)$ & $\sfC_p$ & $\sfB_p$ \\
    \hline 
He.4a & $\phantom{\Big|}$ $\Spin(4p+2)$, $p\geq2$  & $\GL(2p+1)$ & $\sfBC_p$ & $\sfB_p$ \\
\hline
He.4b & $\phantom{\Big|}$ $\Spin(4p)$, $p\geq3$ &  $\GL(2p)$ & $\phantom{p}\sfC_p$ & $\sfB_p$  \\
\hline
He.5 & $\phantom{\Big|}$ $\sfE_6$ & $\sfD_5\times \CC^*$ & $\sfBC_2$ & $\sfB_2$ \\
\hline
He.6 & $\phantom{\Big|}$ $\sfE_7$ &  $\sfE_6 \times \CC^*$ & $\sfC_3$ & $\sfB_3$ \\
\hline
\end{tabular}
\end{table}

%\begin{table}[H]
%\begin{tabular}{|c||cc|c|c|c|c|c|}
%\hline 
%& $\phantom{\bigg|}$ $G/K$ & & $\Phi_{G/K}$ & $m_{\tfrac{1}{2}\ell}$ & $m_s$ & $m_{\ell}$ \\
%    \hline    \hline
%H1.a & $\phantom{\Big|}$ $\SL(p+q)/\mathrm S(\GL(p) \times \GL(q))$ & \!\!($p < q$) \! & $\sfBC_p$ & $\!2(q-p)\!$ & $2$ & $1$ \\
%    \hline
%H1.b & $\phantom{\Big|}$ $\SL(2p)/\mathrm S(\GL(p) \times \GL(p))$ & & $\sfC_p$ & 0 & 2 & 1 \\
%    \hline
%H2 & $\phantom{\Big|}$ $\SO(n)/\SO(2) \times \SO(n-2)$ & & $\sfC_2$ & 0 & $n-4$ & 1 \\
%    \hline 
%H3 & $\phantom{\Big|}$ $\Sp(2p)/\GL(p)$ & & $\sfC_p$ & 0 & 1 & 1 \\
%    \hline 
%H4.a & $\phantom{\Big|}$ $\SO(2p)/\GL(p)$ & \!\!($p$ odd) \!& $\phantom{p/2} \sfBC_{\lceil p/2\rceil}$ & 4 & 4 & 1 \\
%%\tfrac{p+1}{2}
%\hline
%H4.b & $\phantom{\Big|}$ $\SO(2p)/\GL(p)$ & \!\!($p$ even)\! & $\phantom{p}\sfC_{p/2}$ & 0 & 4 & 1 \\
%\hline
%H5 & $\phantom{\Big|}$ $\sfE_6/\sfD_5\times \CC^*$ & & $\sfBC_2$ & 8 & 6 & 1 \\
%\hline
%H6 & $\phantom{\Big|}$ $\sfE_7/\sfE_6 \times \CC^*$ & & $\sfC_3$ & 0 & 8 & 1 \\
%\hline
%\end{tabular}
%\caption{Symmetric varieties of Hermitian type}
%\end{table}

\begin{itemize}
	\item[i)] $\mathbf{SL(p+q)/\mathrm S( GL(p) \times GL(q))}$, $q > p \geq 1$: $\Delta_X = \{\alpha_i + \alpha_{p+q-i}\ :\ 1\leq i < p \}\cup\{\alpha_p + \ldots + \alpha_q \}$, $\Lambda^+_X$ is generated by the weights $\omega_i + \omega_{p+q-i}$ with $1\leq i\leq p$;
	\item[ii)] $\mathbf{SL(2p)/\mathrm S(GL(p) \times GL(p))}$, $p \geq 1$: $\Delta_X = \{\alpha_i + \alpha_{2p-i}\ :\ 1\leq i < p\}\cup\{\alpha_p\}$, $\Lambda^+_X$ is generated by the weights $\omega_i + \omega_{2p-i}$ with $1\leq i < p$, toghether with $2 \omega_p$;
	\item[iii)] $\mathbf{Spin(2p+1)/Spin(2) \times Spin(2p-1)}$, $p \geq 2$: $\Delta_X = \{\alpha_1, \; 2 (\alpha_2 + \ldots + \alpha_p)\}$, $\Lambda^+_X$ is generated by $2\omega_1$ and $\omega_2$;
	\item[iv)] $\mathbf{Spin(2p)/Spin(2) \times Spin(2p-2)}$, $p \geq 3$: $\Delta_X = \{\alpha_1, \; 2 (\alpha_2 + \ldots + \alpha_{p-2}) + \alpha_{p-1}+\alpha_p \}$, $\Lambda^+_X$ is generated by $2\omega_1$ and $\omega_2$;
	\item[v)] $\mathbf{Sp(2p)/GL(p)}$, $p \geq 2$: $\Delta_X =\{2\alpha_i\ :\ 1\leq i\leq p-1\}\cup\{\alpha_p\}$, $\Lambda^+_X$ is generated by the weights $2\omega_i$ with $1 \leq i \leq p$;
	\item[vi)] $\mathbf{Spin(4p+2)/GL(2p+1)}$, $p \geq 1$:  $\Delta_X =\{\alpha_{2i-1}+2\alpha_{2i}+\alpha_{2i+1}\ :\ 1\leq i < p\}\cup\{\alpha_{2p-1} + \alpha_{2p} + \alpha_{2p+1} \}$, $\Lambda^+_X$ is generated by the weights $\omega_{2i}$ with $1 \leq i < p$ together with $\omega_{2p} + \omega_{2p+1}$;
	\item[vii)] $\mathbf{Spin(4p)/GL(2p)}$, $p \geq 2$:  $\Delta_X =\{\alpha_{2i-1}+2\alpha_{2i}+\alpha_{2i+1}\ :\ 1\leq i\leq p-1\}\cup\{ \alpha_{2p} \}$, $\Lambda^+_X$ is generated by the weights $\omega_{2i}$ with $1 \leq i < p$, toghether with $2 \omega_{2p}$ ;
	\item[viii)] $\mathbf{E_6/D_5 \times \CC^*}$: $\Delta_X = \{\alpha_1 + \alpha_3 +\alpha_4+\alpha_5 + \alpha_6, \; 2\alpha_2 +  \alpha_3 + 2\alpha_4 + \alpha_5 \}$, $\Lambda^+_X$ is generated by $\omega_1+\omega_6$ and $\omega_2$;
	\item[ix)] $\mathbf{E_7/E_6\times \CC^*}$: $\Delta_X = \{2\alpha_1 + \alpha_2 + 2\alpha_3 +2\alpha_4 + \alpha_5, \; \alpha_2 +  \alpha_3 + 2(\alpha_4 + \alpha_5 + \alpha_6), \; \alpha_7\}$, $\Lambda^+_X$ is generated by $\omega_1$, $\omega_6$ and $2\omega_7$.
	\end{itemize}

If $\widetilde \Phi_X$ is non-reduced, then $K$ is self-normalizing, and the derived subgroup $K'$ is also spherical in $G$. If instead $\widetilde \Phi_X$ is reduced, then $K$ is not self-normalizing and its derived subgroup is not spherical in $G$. Therefore, there are some Hermitian symmetric subgroups which have a non-symmetric semisimple spherical counterpart, similarly to the cases Sph.A7 and Sph.A9 of Table \ref{tab:sph_pairs_A}.

We list such spherical pairs in Table \ref{tab:hermitiani derivati sferici}. In all cases, we also give the simple restricted roots and a minimal set of generators for the weight monoid $\Lambda^+_X$.  

\begin{table}[H]
\caption{Non-symmetric reductive spherical pairs $(G,K)$ with $G/\rmN_G(K)$ symmetric of Hermitian type}
\label{tab:hermitiani derivati sferici}
\begin{tabular}{|c|c|c|c|c|c|c|c|}
\hline 
& $\phantom{\bigg|}$ $G$  & $K$ & $\Phi_{G/K}$ \\
    \hline    \hline
He.1a' & $\phantom{\Big|}$ $\SL(p+q)$, $q>p\geq 1$  & $\mathrm \SL(p) \times \SL(q)$ & $\sfB_p$ \\
    \hline
He.4a' & $\phantom{\Big|}$ $\Spin(4p+2),$ $p\geq 1$&  $\SL(2p+1)$ & $\sfB_p$ \\
\hline
He.5' & $\phantom{\Big|}$ $\sfE_6$ & $\sfD_5$ & $\sfB_2$ \\
\hline
\end{tabular}
\end{table}

\begin{itemize}
	\item[i')] $\mathbf{SL(p+q)/SL(p) \times SL(q))}$, $q > p \geq 1$: $\Delta_X = \{\alpha_i + \alpha_{p+q-i}\ :\ 1\leq i < p \}\cup\{\alpha_p + \ldots + \alpha_q \}$, $\Lambda^+_X$ is generated by the weights $\omega_i + \omega_{p+q-i}$ with $1\leq i < p$ together with $\omega_p$ and $\omega_q$;
	\item[vi')] $\mathbf{Spin(4p+2)/SL(2p+1)}$, $p \geq 1$:  $\Delta_X =\{\alpha_{2i-1}+2\alpha_{2i}+\alpha_{2i+1}\ :\ 1\leq i < p\}\cup\{\alpha_{2p-1} + \alpha_{2p} + \alpha_{2p+1} \}$, $\Lambda^+_X$ is generated by the weights $\omega_{2i}$ with $1 \leq i < p$ together with $\omega_{2p}$ and $\omega_{2p+1}$;
	\item[viii')] $\mathbf{E_6/D_5}$: $\Delta_X = \{\alpha_1 + \alpha_3 +\alpha_4+\alpha_5 + \alpha_6, \; 2\alpha_2 +  \alpha_3 + 2\alpha_4 + \alpha_5 \}$, $\Lambda^+_X$ is generated by $\omega_1$, $\omega_2$ and $\omega_6$.
\end{itemize}

If $(G,K')$ is a spherical pair appearing in Table \ref{tab:hermitiani derivati sferici} and $K = \rmN_G(K')$ is the corresponding symmetric subgroup of Hermitian type, then we have $\Omega_{G/K}^+ =\Lambda_{G/K'}^+$ and $\CC[G]^{(K)} = \CC[G]^{K'}$, as in the cases Sph.A7 and Sph.A9. Thus the problem of decomposing the product of two spherical functions on $G/K$ reduces to the analogous problem on $G/K'$.

%There are two other Levi subgroups which are spherical (and not symmetric). 
%
%\begin{table}[H]
%\begin{tabular}{|c||cc|c|c|c|c|c|c|}
%\hline 
%& $\phantom{\bigg|}$ $G$ & & $K$ & $\Phi_{G/K}$ \\
%    \hline    \hline
%SphLevi.7 & $\phantom{\Big|}$ $\SO(2p+1)$ & & $\GL(p)$ & $\sfB_p \times \sfB_p$ \\
%    \hline
%SphLevi.8 & $\phantom{\Big|}$ $\Sp(2n)$ & & $\GL(1) \times \Sp(2n-2)$ & $\sfA_1 \times \sfA_1$ \\
%    \hline 
%\end{tabular}
%\caption{Spherical Levi subgroups which are not symmetric}
%\end{table}
%
%The case SphLevi.8 has already been considered in Section \ref{sec:casi_tipo_A} (case Sph.A11 of Table \ref{tab:sph_pairs_A}).

In the same spirit of Conjectures \ref{conj:stanley_cor} and \ref{conj:affine_spherical_A}, we formulate the following for the multiplication of spherical functions in terms of the tensor product of the restricted root system, which by computational experiments seems to hold for all the Hermitian symmetric pairs but in case He.1a. 

\begin{conjecture}\label{conj:hermitiano}
Let $(G,K)$ be a reductive spherical pair with $G$ simple and $K$ a Levi subgroup of $G$ different from cases He.1a. For all $\lambda, \mu, \nu \in \Lambda_X^+$, it holds
$$
E_{G/K}(\nu) \subset E_{G/K}(\lambda) \cdot E_{G/K}(\mu) \Longleftrightarrow
\left\{ \begin{array}{c}
V_X(\nu) \subset V_X(\lambda) \otimes V_X(\mu)  \phantom{\Big|} \\
\nu \leq_X \lambda + \mu
\end{array} \right.
$$
\end{conjecture}

By Theorem \ref{teo:isogeny-dec} and Corollary \ref{cor:decomposizione}, the previous conjecture immediately reduces to the case where $G/K$ is a symmetric variety of Hermitian type.

%\begin{conjecture} 	\label{conj:Dn_hermitiano}
%Let $X = \SO(2r)/\GL(r)$ and let $\lambda, \mu, \nu \in \Lambda_X^+$. Then
%$$
%E(\nu) \subset E(\lambda) \cdot E(\mu) \Longleftrightarrow
%\left\{ \begin{array}{c}
%V_X(\bar \nu) \subset V_X(\bar \lambda) \otimes V_X(\bar \mu)  \phantom{\Big|} \\
%\bar \nu \leq_X \bar \lambda + \bar \mu
%\end{array} \right.
%$$
%\end{conjecture}

\begin{remark}
 In \cite[Section 9]{BGM} the following question was considered. Suppose that $G$ is simply laced and let $X = G/K$ be a spherical variety with $\Delta^\rmsc_X \cap \Delta= \varnothing$. Let $\lambda, \mu, \nu \in \Lambda_X^+$, is it true that $E_X(\nu) \subset E_X(\lambda) \cdot E_X(\mu)$ if and only if $V(\nu) \subset V(\lambda)  \otimes V(\mu)$ and $\nu \leq_{\Delta^\rmsc_X} \lambda + \mu$?

This indeed seems to be true in several cases, and in particular in case He.1a. Assuming the validity of Conjecture \ref{conj:stanley_cor}, it is easy to see that when $X$ is symmetric and both $\Phi$ and $\Phi_X$ are of type $\sfA$ a positive answer to the previous question follows from the saturation theorem \cite{KT}. See also \cite{GH}, where a similar question is considered for the symmetric varieties with restricted root system of type $\sfA$. However, the following example shows that the previous question does not always have a positive answer, at least when $\widetilde\Phi_X$ is non-reduced. Conjecture \ref{conj:hermitiano} is a partial attempt to fix the question.
\end{remark}

\begin{example} \label{ex:counterexample}
Consider the symmetric variety $X = \Spin(10)/\GL(5)$. The restricted root system $\widetilde \Phi_X$ is of type $\sfBC_2$, with base
$$
	\widetilde \Delta_X = \{\alpha_1 + 2\alpha_2 + \alpha_3, \; \alpha_3 + \alpha_4 + \alpha_5 \}.
$$
Let us consider the highest root $\theta = \alpha_1 + 2\alpha_2 + 2\alpha_3 + \alpha_4 + \alpha_5$: then $\theta \in \Lambda_X^+$ and $\theta \leq_X 2 \theta$. Moreover $V(\theta) \subset V(\theta)^{\otimes 2}$, however taking the product inside $\CC[X]$ we have $E_X(\theta) \not\subset E_X(\theta)^2$: indeed the unique non-zero equivariant projection
$$
	\pi : V(\theta) \otimes V(\theta) \rightarrow V(\theta)
$$
is given by the bracket on the Lie algebra $\mathfrak g \simeq V(\theta)$. It follows that $\pi(v \otimes v) = 0$ for all $v \in V(\theta)$, and we conclude by Proposition \ref{prop:supporto_invarianti}.

The same happens more generally for $\Spin(2p)/\GL(p)$ whenever $p$ is odd, and for $\sfE_6/\sfD_5 \times \CC^*$: indeed in all these cases the highest root is a spherical weight, and $V(\theta)$ occurs in $ V(\theta)^{\otimes 2}$ with multiplicity one. 

Both these examples are instances of exceptional symmetric varieties of Hermitian type, that is, symmetric varieties of Hermitian type with non-reduced restricted root system. There is another instance of such a symmetric variety, namely $X = \SL(p+q)/\mathrm S(\GL(p) \times \GL(q))$ when $p \neq q$: however in this case $V(\theta)$ occurs in $ V(\theta)^{\otimes 2}$ with multiplicity two, and the inclusion $E_X(\theta) \subset E_X(\theta)^2$ holds true inside $\CC[X]$.
\end{example}

\end{document}